\documentclass[sn-mathphys,Numbered]{sn-jnl}

\usepackage{graphicx}%
\usepackage{multirow}%
\usepackage{amsmath,amssymb,amsfonts}%
\usepackage{amsthm}%
\usepackage{mathrsfs}%
\usepackage[title]{appendix}%
\usepackage{xcolor}%
\usepackage{textcomp}%
\usepackage{manyfoot}%
\usepackage{booktabs}%
\usepackage{algorithm}%
\usepackage{algorithmicx}%
\usepackage{algpseudocode}%
\usepackage{listings}%

\theoremstyle{thmstyleone}%
\theoremstyle{thmstyletwo}%

\theoremstyle{thmstylethree}%

\RequirePackage[OT1]{fontenc}
\usepackage{amsthm,amsmath}
\usepackage{amssymb}
\RequirePackage{graphicx}
\RequirePackage[numbers]{natbib}

\usepackage{amssymb}
\usepackage[T1]{fontenc}
\usepackage[latin1]{inputenc}
\usepackage[english]{babel}
\usepackage{amsfonts}
\usepackage{enumitem}
\usepackage{amsmath}
\usepackage{amsthm}
\usepackage{booktabs}
\usepackage[normalem]{ulem}
\usepackage{rotating}
\usepackage{amssymb}
\usepackage{tabularx,ragged2e,booktabs,caption}
\usepackage{multirow}
\usepackage{subcaption}
\usepackage{stmaryrd}
\usepackage{url}
\usepackage{hyperref}

\numberwithin{equation}{section}
\theoremstyle{plain}

\def\R{\mathbb{R}}

\def\N{\mathbb{N}}
\def\P{\mathbb{P}}
\def\E{\mathbb{E}}
\def\L{\mathbb{L}}

\def\R{\mathbb{R}}
\def\C{\mathbb{C}}
\def\Z{\mathbb{Z}}

\def\1{\mbox{I\hspace{-.6em}1}} 

\def\cov{\mbox{Cov}\,}
\def\var{\mbox{Var}\,}

\def\1{\mbox{\hspace{.2em}I\hspace{-.6em}1}} 

\def\limiteasn{\renewcommand{\arraystretch}{0.5}
\begin{array}[t]{c}\stackrel{a.s.}{\longrightarrow} \\
{\scriptstyle
n\rightarrow \infty}\end{array}\renewcommand{\arraystretch}{1}}

\def\limiten{\renewcommand{\arraystretch}{0.5}
\begin{array}[t]{c}\stackrel{}{\longrightarrow} \\
{\scriptstyle
n\rightarrow \infty}\end{array}\renewcommand{\arraystretch}{1}}

\def\limiteloin{\renewcommand{\arraystretch}{0.5}
\begin{array}[t]{c}\stackrel{{\cal L}}{\longrightarrow} \\
{\scriptstyle
n\rightarrow \infty}\end{array}\renewcommand{\arraystretch}{1}}

\def\limiteproban{\renewcommand{\arraystretch}{0.5}
\begin{array}[t]{c}\stackrel{{\P}}{\longrightarrow} \\
{\scriptstyle
n\rightarrow \infty}\end{array}\renewcommand{\arraystretch}{1}}

\def\limiteL1{\renewcommand{\arraystretch}{0.5}
\begin{array}[t]{c}\stackrel{{\L^1}}{\longrightarrow} \\
{\scriptstyle n\rightarrow+\infty}\end{array}\renewcommand{\arraystretch}{1}}

\def\equin{\renewcommand{\arraystretch}{0.5}
\begin{array}[t]{c}\stackrel{}{\sim} \\
{\scriptstyle
n\rightarrow \infty}\end{array}\renewcommand{\arraystretch}{1}}

\def\equix{\renewcommand{\arraystretch}{0.5}
\begin{array}[t]{c}\stackrel{}{\sim} \\
{\scriptstyle
x\rightarrow \infty}\end{array}\renewcommand{\arraystretch}{1}}

\theoremstyle{plain} 
\theoremstyle{remark}
\DeclareMathOperator*{\cart}{\times}

  \newtheorem{prop}{Proposition}[section]
  \newtheorem{cor}{Corollary}[section]
\newtheorem{theo}{Theorem}[section]
 \newtheorem{lem}{Lemma}[section]
 
  \newtheorem{Rem}{Remark}[section]

\begin{document}

\title[Quasi-Maximum Likelihood Estimation of long-memory linear processes]{Quasi-Maximum Likelihood Estimation of long-memory linear processes}


\author*[1]{\fnm{Jean-Marc} \sur{Bardet}}\email{bardet@univ-paris1.fr}

\author[1]{\fnm{Yves Gael} \sur{Tchabo MBienkeu}}\email{tch.yves@yahoo.fr}
\equalcont{These authors contributed equally to this work.}

\affil*[1]{\orgdiv{SAMM}, \orgname{University Panth\'eon-Sorbonne}, \orgaddress{\street{90 rue Tolbiac}, \city{Paris}, \postcode{75013}, \country{France}}}


\abstract{The purpose of this paper is to study the convergence of the quasi-maximum likelihood (QML) estimator for long memory linear processes. We first establish a correspondence between the long-memory linear process representation and the long-memory AR$(\infty)$ process representation. We then establish the almost sure consistency and asymptotic normality of the QML estimator. Numerical simulations illustrate the theoretical results and confirm the good performance of the estimator.}

\keywords{Long memory process, Semiparametric estimation, Linear process, Limit theorems}



\maketitle

\section{Introduction}\label{sec1}

Since Hurst's (1953) introduction of long-range dependent processes, much research has focused on estimating the long-range parameter, whether defined on the basis of the asymptotic power-law behavior of the correlogram at infinity or that of the spectral density at zero (see the monographs \cite{Be} and \cite{DOT} for more details). \\
Two estimation frameworks have been studied extensively. The first focused on the estimation of the long-memory parameter alone, but could be carried out in a semi-parametric framework, {\i.e.} if only the asymptotic behavior of the correlation or spectral density was specified. This led to the first methods proposed historically, such as those based on the R/S statistic, on quadratic variations, on the log-periodogram, or more recent methods such as wavelet or local Whittle (again, see \cite{DOT} for more details). \\
Here we are interested in a more parametric framework, and in estimating all the parameters of the process, not just the long memory parameter. The first notable results on the asymptotic behavior of such a parametric estimator were obtained in Fox and Taqqu (1986) (see \cite{FT}) in the special case of Gaussian long-memory processes, using the Whittle estimator. These results were extended to linear long-memory processes with a moment of order 4 by Giraitis and Surgailis (1990) (see \cite{gi-sur}). In both settings, the asymptotic normality of the estimator was proved, while non-central limit theorems were obtained for functions of Gaussian processes in \cite{GT} or for increments of the Rosenblatt process in \cite{BT}. The asymptotic normality of the maximum likelihood estimator for Gaussian time series was also obtained by Dahlhaus (1989, see \cite{Da}) using that of the Whittle estimator obtained in Fox and Taqqu (1986). \\
For weakly dependent time series, especially for conditionally heteroscedastic processes such as GARCH processes, the quasi-maximum likelihood (QML) estimator has become the benchmark for parametric estimation, providing very interesting convergence results where Whittle's estimator would not. This is true for GARCH or ARMA-GARCH processes (see \cite{BH} and \cite{FZ}), but also for many others such as ARCH($\infty$), AR$(\infty)$, APARCH processes, etc. (see \cite{BW}). We will also note convergence results for this modified estimator for long-memory squares processes, typically LARCH($\infty$) processes, see \cite{Beran}, or quadratic autoregressive conditional heteroscedastic processes, see \cite{DGS}. But for long-memory processes, such as those defined by a non-finite sum of their autocorrelations, to our knowledge only the paper by Boubacar Mainassara {\it et al.} (2021) (see \cite{BES}) has shown the normality of this QML estimator in the special case of a FARIMA$(p,d,q)$ process with weak white noise. \\
We therefore propose here to study the convergence of the Gaussian QML estimator in the general framework of long-memory one-sided linear processes. In such a framework, we begin by noting that the QML estimator is in fact a non-linear least-squares estimator. The key point of our approach is to prove that long-memory one-sided linear processes can be written in autoregressive form with respect to their past values, which we can call long-memory linear AR$(\infty)$. This is perfectly suited to the use of QMLE, since this estimator is obtained from the conditional expectation and variance of the process. We then show the almost sure convergence of QMLE for these long-memory AR$(\infty)$ processes, which generalizes a result obtained in \cite{BW} for weakly dependent AR$(\infty)$ processes. We also prove the asymptotic normality of this estimator, which provides an alternative to the asymptotic normality of Whittle's estimator obtained in \cite{gi-sur}. An advantage of QML estimation lies in the fact that, because it is applied to processes with an AR$(\infty)$ representation, the fact that the $(X_t)$ series is centered or not has no effect at all on the parameters of this AR$(\infty)$ representation, particularly on the estimation of the long memory parameter. \\
Finally, we performed simulations of two long-memory time series and examined the performance of the QMLE as a function of the size of the observed trajectories. This showed that the behavior of the QMLE is consistent with theory as the size of the trajectories increases, and provides a very accurate alternative to Whittle's estimator. An application on real data (average monthly temperatures in the northern hemisphere) is also presented. \\
~\\
This article is organized as follows: the section \ref{Affine} below presents the AR$(\infty)$ notation of an arbitrary long memory one-sided linear process, the section \ref{SecQMLE} is devoted to the presentation of the QMLE estimator and its asymptotic behavior, numerical applications are treated in the section \ref{simu}, while all proofs of the various results can be found in the section \ref{proofs}.

\

\section{Long-memory linear causal time series} \label{Affine}
Assume that $\varepsilon=(\varepsilon_t)_{t \in \mathbb{Z}}$ is a sequence of centered independent random variables such as $\mathbb{E}[\varepsilon_0^2]=1$ and 
$(a_i)_{i \in \mathbb{N}}$ is a sequence of real numbers such as: 
\begin{equation}\label{a2} 
a_{i} = L_a(i)\, i^{d-1}\quad \mbox{for $i\in \N^*$ and}\quad a_0>0,
\end{equation} 
where $d\in (0,1/2)$ and with $L_{a}(\cdot)$ a positive slow varying function satisfying 
$$
\mbox{for any $t>0$,}\quad 
\lim_{x\to \infty} \frac{L_a(xt)}{L_a(x)}=1.
$$
Now, define the causal linear process $(X_t)_{t \in \mathbb{Z}}$ by
\begin{equation}\label{linear} 
X_t = \sum_{i=0}^{\infty}{a_i \, \varepsilon_{t-i}}\quad\mbox{for any $t \in \mathbb{Z}$}.
\end{equation} 
Since $0<d<1/2$, it is well know that $(X_t)_{t\in \Z}$ is a second order stationary long-memory process. Indeed, its autocovariance is 
\begin{equation}\label{cov} 
r_X(k)=\cov(X_0,X_k)=\sum_{i=0}^{\infty}{a_i\, a_{i+k}}\sim C_d\, L_{a}^{2}(k)\, k^{2d-1}\quad \mbox{when $k \to \infty$},
\end{equation} 
where $C_d=\int_0^\infty (u+u^2)^{d-1}\,du$ (see for instance Wu {\it et al.}, 2010). \\
~\\
Then, it is always possible to provide a causal affine representation for $(X_t)_{t\in \Z}$, {\it i.e .} it is always possible to write $(X_t)_{t\in \Z}$ as an AR$(\infty)$ process:\\
\begin{prop}\label{prop1}
Let $(X_t)_{t\in \Z}$ be a causal linear process defined in \eqref{linear} where $(a_i)$ satisfies \eqref{a2}. Then, there exists a sequence of real number $(u_i)_{i\in \N^*}$ such as:
\begin{equation}\label{causal} 
X_t=a_0\, \varepsilon_t+\sum_{i=1}^\infty u_i \, X_{t-i}\quad \mbox{for any $t\in \Z$},
\end{equation}
where $(u_i)_{i\in \N^*}$ satisfies 
\begin{equation}\label{u}
\sum_{i=1}^\infty u_i=1\quad\mbox{and}\quad u_n\equin  \frac{a_0\, d}{\Gamma(d)\, \Gamma(1-d)}\, L^{-1}_a(n) \, n^{-1-d} =L_u(n)\, n^{-1-d}
\end{equation}
where $L_u$ is a slow varying function.
\end{prop}

~\\
\begin{Rem}\label{Rem000}
Using \eqref{prod}, the reciprocal implication of Proposition \ref{prop1} is also true:  if $(X_t)$ satisfies the linear affine causal representation \eqref{causal} 
where $(u_i)_{i\in \N}$ satisfies \eqref{u}, then $(X_t)$ is a one-sided long-memory linear process satisfying  \eqref{linear} where $(a_i)$ satisfies \eqref{a2}.
\end{Rem}
~\\

\begin{Rem}
It is also known that $\displaystyle \Gamma(d)\, \Gamma(1-d)=\frac \pi {\sin(\pi \, d)}$ for any $d\in (0,1)$, and this  implies $\displaystyle u_n\equin  \frac{a_0\, d \, \sin(\pi \, d)}{\pi\, L_a(n)} \, n^{-1-d}$.\\
\end{Rem}
As a consequence, every long-memory one-sided linear process is a long-memory AR$(\infty)$ process with the special property that the sum of the autoregressive coefficients equals $1$. This is the key point for the use of quasi-maximum likelihood estimation in the following section. \\
~\\
{\bf Example of the FARIMA process:} Let $(X_t)_{t \in \mathbb{Z}}$ be a standard FARIMA($0,d,0$) with $d\in (0,1/2)$, which means $X=(I-B)^{-d} \varepsilon$, where $B$ is the usual backward linear operator on $\R^\Z$ and $I$ the identity operator. Then, using the power series of $(1-x)^{-d}$, it is known that
$$
X_t=\sum_{i=0}^{\infty}{a_{i}\varepsilon_{t-i}}\quad \mbox{with}\quad a_{i} = \frac{\Gamma(i+d)}{\Gamma(i+1)\Gamma(d)}\quad\mbox{for $t\in \Z$}.
$$
Using the Stirling expansion of the Gamma function, {\it i.e.} $\Gamma(x)\equix \sqrt{2\pi}\, e^{-x}x^{x-1/2}$, we obtain $a_{n}  \equin \frac{1}{\Gamma(d)}\, n^{d-1}$, which is \eqref{a2} with $L_a(n)\equin \frac{1}{\Gamma(d)}$. \\ 
Moreover, the decomposition $X=\varepsilon+(I_d- (I_d-B)^d)\, X$ implies:
$$
X_t=\varepsilon_t + d \,  \sum_{i=1}^\infty \frac { \Gamma(i-d)}{\Gamma(1-d) \, \Gamma(i+1)} \, X_{t-i}\quad\mbox{for $t\in \Z$}.
$$
The expansion $\displaystyle \frac { \Gamma(n-d)}{\Gamma(n+1)}\equin  n^{-1-d}$ provides $X_t=\varepsilon_t + \sum_{n=1}^\infty u_n\, X_{t-n}$ with  $u_n\equin \frac {d}{\Gamma(1-d)}\, n^{-1-d}$ is equivalent to \eqref{u} when $a_0=1$ and $L_a(n)\equin \frac{1}{\Gamma(d)}$.

\section{Asymptotic behavior of the Gaussian Quasi-Maximum Likelihood Estimator}\label{SecQMLE}
\subsection{Definition of the estimator}
We will assume that $(X_t)_{t\in \Z}$ is a long-memory one-sided linear process written as an AR$(\infty)$ process, {\it i.e.}
\begin{equation}\label{LRD7}
X_t=\sigma^* \, \varepsilon_t + \sum_{k=1}^\infty u_k(\theta^*)\, X_{t-k}\quad \mbox{for any $t\in \Z$,}
\end{equation}
where 
\begin{itemize} 
\item $(\varepsilon_t)_{t\in \Z}$ is a white noise, such that $\varepsilon_0$ has an absolutely continuous probability measure with respect to the Lebesgue measure and such that $\E[\varepsilon_0^2]=1$;
\item for $\theta={}^t(\gamma,\sigma^2) \in \Theta$ a compact subset of $\R^{p-1}\times (0,\infty)$, 
$(u_n(\theta))_{n\in \N}$  is a sequence of real numbers satisfying for any $\theta \in \Theta$,
\begin{equation}\label{untheta}
u_n(\theta)= L_\theta(n)\,n^{-d(\theta)-1}\quad\mbox{for $n \in \N^*$ and}\quad\sum_{n=1}^\infty u_n(\theta)=1.
\end{equation}
with $d(\theta) \in (0,1/2)$. We also assume that 
{\bf the sequence $(u_n(\theta))$ does not depend on $\sigma^2$};
\item $\theta^*={}^t(\gamma^*,\sigma^{*2})$, $\theta^*$ is in the interior of $\Theta$, with $\sigma^*>0$ an unknown real parameter and $\gamma^* \in \R^{p-1}$ an unknown vector of parameters. 
\end{itemize} 
~\\
A simple example of such a sequence $(u_n(\theta))$ is $u_n(\theta)=(\zeta(1+d))^{-1}\, n^{-1-d}$ for $n \in \N^*$, with $\theta=(d,\sigma^2)\in (0,1/2)\times (0,\infty)$ where $\zeta(\cdot)$ is the  Riemann zeta function. Then $\Theta= [d_m,d_M] \times[\sigma_m^2,\sigma_M^2]$, with  $0<d_m<d_M<1/2$  and $0<\sigma_m^2<\sigma_M^2$.\\
~\\
For ease of reading, denote $d^*=d(\theta^*)$ the long-memory parameter of $(X_t)$. Denote also $d^*_+=d^*+\varepsilon$ where $\varepsilon \in (0,1/2-d^*)$ is chosen as small as possible. Since $(u_n(\theta))_{n\in \N}$ satisfies \eqref{untheta}, we know from Remark \ref{Rem000} that there exists $C_a$ such that for any $t\in\Z$,
\begin{equation}\label{linlin}
X_t=\sum_{i=0}^\infty a_i(\theta^*)\, \varepsilon_{t-i}\quad \mbox{with}\quad |a_i(\theta^*)|\leq \frac {C_a}{i^{1-d^*_+}}~~\mbox{for all $i\in \N^*$}.
\end{equation}
We also deduce from \eqref{cov} that there exists $C_c>0$ satisfying
\begin{equation}\label{cov2}
|r_X(k)|=\big |\cov(X_0,X_k)\big | \leq \frac {C_c}{(1+k)^{1-2d^*_+}}~~\mbox{for all $k\in \N$}.
\end{equation}
In the sequel we will also denote for any $\theta \in \Theta$,
\begin{equation}\label{mt}
m_t(\theta)=\sum_{k=1}^\infty u_k(\theta)\, X_{t-k}\quad \mbox{for any $t\in \Z$}.
\end{equation}
We want to estimate $\theta^*$ from an observed trajectory $(X_1,\ldots,X_n)$, where $(X_t)$ is defined by \eqref{LRD7}. For such an autoregressive causal process, a Gaussian quasi-maximum likelihood estimator is really appropriate, since it is built on the assumption that $(\varepsilon_t)$ is Gaussian white noise, and it is well known that an affine function of $\varepsilon_t$ is still a Gaussian random variable (see for example Bardet and Wintenberger, 2009). It consists in considering the log-conditional density $I_n(\theta)$ of $(X_1,\ldots,X_n)$ when $(\varepsilon_t)$ is a standard Gaussian white noise and with $X_t=\sigma \, \varepsilon_t + m_t(\theta)$,  {\it i.e.}
\begin{equation}\label{I}
I_n(\theta)=\sum_{t=1}^n q_t(\theta)=-\frac{1}{2}\sum_{t=1}^{n}\left(\log \big(\sigma^2\big )
+ \frac{\big (X_{t}-m_t(\theta)\big)^{2}}{\sigma^2}\right)\quad\mbox{for any $\theta \in \Theta$}.
\end{equation}
However, such conditional log-likelihood is not a feasible statistic since  $m_t(\theta)$  depends on $(X_k)_{k\leq 0}$ which is unobserved. Hence it is usual to replace  $m_t(\theta)$ by the following approximation:
\begin{equation}\label{mth}
\widehat m_t(\theta)=\sum_{i=1}^{t-1} u_i(\theta)\, X_{t-i} \quad \mbox{for any $t\in \N^*$},
\end{equation}
with the convention $\sum_{t=1}^0=0$. Then, a quasi conditional log-likelihood $\widehat I_n(\theta)$ can be defined:
\begin{equation}\label{QLik}
\widehat I_n(\theta)=-\frac{1}{2}\sum_{t=1}^{n}\left(\log \big(\sigma^2\big ) + \frac{\big (X_{t}-\widehat m_t(\theta)\big)^{2}}{\sigma^2}\right).
\end{equation}
If $\Theta $ is a subset of $\R^p$ such as for all $\theta \in \Theta$ there exists an almost surely stationary solution of the equation $X_t=\sigma \, \varepsilon_t + m_t(\theta)$ for any $t\in \Z$, we define the Gaussian quasi maximum likelihood estimator (QMLE) of $\theta$ by 
\begin{equation}\label{QMLE}
\widehat \theta_n= \text{Arg}\!\max_{\!\!\!\!\!\!\theta \in \Theta} \widehat I_n(\theta).
\end{equation}
Note that a direct implication of the assumption that $(u_n(\theta))$ does not depend on $\sigma^2$ is that if we denote $\widehat \theta_n={}^t(\widehat \gamma_n,\widehat \sigma_n^2)$ the QMLE, then:
$$
\widehat \gamma_n=\text{Arg}\!\min_{\!\!\!\!\!\!\!\!\!\!\!(\gamma,\sigma^2) \in \Theta}\sum_{t=1}^n \big (X_t -\sum_{k=1}^{t-1} u_k(\gamma)\, X_{t-k}\big )^2~ \mbox{and}~ \widehat \sigma_n^2=\frac 1 n \, \sum_{t=1}^n \big (X_t -\sum_{k=1}^{t-1} u_k(\widehat \gamma_n)\, X_{t-k}\big )^2,
$$
where by writing convention $u_n( \theta)=u_n(\gamma)$. Hence, in this case of long-memory AR$(\infty)$, $\widehat \gamma_n$ is also a {\bf non-linear least square estimator} of the parameter $\gamma$.  

\subsection{Consistency and asymptotic normality of the estimator}

The consistency of the QMLE is established under additional assumptions. \\

\begin{theo}\label{Theo1}
Let $(X_t)_{t\in \Z}$ be a process defined by \eqref{LRD7} and its assumptions. Assume also:
\begin{itemize}
\item for any $n\in \N^*$, $\theta \in \Theta \mapsto u_n(\theta)$ is a continuous function on $\Theta$; 
\item If $u_n(\theta)=u_n(\theta')$ for all $n\in \N^*$ with $\theta=(\gamma,\sigma^2)$ and $\theta'=(\gamma',\sigma^{2})$, then $\theta=\theta'$.  
\end{itemize}
Let $\widehat \theta_n$ be the QMLE defined in \eqref{QMLE}. Then
$$
\widehat{\theta}_{n} \limiteasn \theta^{*}.
$$
\end{theo}
This result extends the $\widehat \theta_n$ consistency obtained in Bardet and Wintenberger (2009) to short-memory time series models, including ARMA, GARCH, and APARCH, among others, including AR($\infty$) processes. It also applies to long memory AR($\infty$) processes. \\

\begin{Rem}
Regarding the long-memory linear process example, $\theta^*$ could also be estimated using Whittle's estimator, which is constructed from the spectral density and second-order moments of the process. The consistency and asymptotic normality of this estimator were shown by Giraitis and Surgailis (1990). 
\end{Rem}
~\\
~\\
\noindent Having shown the consistency, we would like to show the asymptotic normality of the QML estimator in the case of the long-memory one-sided linear processes considered above.  This amounts to proving it for linear processes whose linear filter depends on a vector of parameters. This will be the case, for example, for FARIMA$(p,d,q)$ processes, for which Boubacar {\it et al.} (2021) \cite{BES} have already shown asymptotic normality in the more general case where $(\varepsilon_t)$ is weak white noise, {\it i.e.} in the case of weak FARIMA$(p,d,q)$ processes.  \\
~\\
As it is typical to establish the asymptotic normality of an M-estimator, we make assumptions about the differentiability of the sequence of functions $(u_n(\theta))_{n\in \N^*}$ with respect to $\theta$: \\
\begin{itemize}
\item[\textbf{(A)}] \textbf{Differentiability of $(u_n(\theta))_{n\in \N^*}$}: for any $n \in \N^*$, the function $u_n(\theta)$ is a ${\cal C}^2(\Theta)$ function and for any $\delta>0$, there exists $C_\delta>0$ such that:
\begin{equation}\label{Diff}
\sup_{n \in \N} \sup_{\theta \in \Theta}\Big \{  n^{1+d(\theta)-\delta} \, \Big (  \big | u_n(\theta) \big |+ \big \| \partial_{\theta} u_n(\theta) \big \| + \big \| \partial^2_{\theta^2} u_n(\theta) \big \| \Big ) \Big \} \leq C_\delta. 
\end{equation}
Moreover we assume that:
\begin{equation}\label{Var}
\mbox{for $v \in \R^{p-1 }$, if for all $k\in \N^*$,}~~{}^t v \, \partial_{\gamma} u_{k}(\theta^*)= 0\quad \Longrightarrow \quad v=0.
\end{equation}
\end{itemize}
{\bf Example} (called {\bf LM} in the numerical applications): For the simple example where $(u_n(\theta))$ is such as $u_n(\theta)=\zeta(1+d)^{-1} n^{-1-d}$ for $n \in \N^*$ with $\theta=(d,\sigma^2)\in  (0,1/2)\times (0,\infty)$, we have:
\begin{eqnarray*}
 \partial_{d} u_n(\theta) &=&-\frac {n^{-1-d}} {\zeta^2(1+d)}\, \big (\zeta(1+d)\, \log(n)+\zeta'(1+d) \big) \\
  \partial^2_{d^2} u_n(\theta)&=&\frac {n^{-1-d}} {\zeta^3(1+d)}\,\big (\zeta^2(1+d)\, \log^2(n)+2\zeta'(1+d)\zeta(1+d)\, \log(n) \\
  && \hspace{4cm}+2 (\zeta'(1+d))^2-\zeta''(1+d)\zeta(1+d)  \big).
\end{eqnarray*}
Therefore \eqref{Diff} of {\bf(A)} is satisfied with $d(\theta)=d$ (note also that $\delta=0$ is not possible). Moreover \eqref{Var} is also clearly satisfied.
~\\
~\\
\begin{theo}\label{theo2}
Consider the assumptions of Theorem \ref{Theo1} and also that $\E[\varepsilon_0^3]=0$ and $\mu_4=\|\varepsilon_0\|_4^4<\infty$.  Then with $\widehat \theta_n$ defined in \eqref{QMLE},  and if \textbf{(A)} holds,
\begin{multline}\label{TLC2}
\sqrt n \, \big (\widehat \theta_n -\theta^*  \big )=\sqrt n \, \left ( \Big ( \begin{array}{c} \widehat \gamma_n \\
\widehat \sigma^2_n \end{array} \Big ) -\Big ( \begin{array}{c} \gamma^*\\
\sigma^{*2} \end{array} \Big )  \right ) \\
\limiteloin {\cal N} \left ( 0 \, , \,\Big (\begin{array}{cc}
(M^*)^{-1} & 0 \\
0 & \sigma^{*4} \, (\mu^*_4-1)
\end{array} \Big )  \right ), 
\end{multline} 
where $M^*=\frac 1 {\sigma^{*2}} \,\sum_{k=1}^\infty \sum_{\ell=1}^\infty \partial_\gamma u_k((\gamma^*,0))\, {}^t \big (\partial_\gamma u_\ell((\gamma^*,0)) \big )\, r_X(\ell-k)$.\\
\end{theo}
It is clear that $\widehat \theta_n$ satisfies \eqref{TLC2} in the case of the FARIMA processes (but this asymptotic normality has been already established under more general assumptions in Boubacar Maïnassara {\it et al.}, 2021, \cite{BES}) or in the case of the {\bf LM} processes example. It is also worth noting that the central limit theorem is written in exactly the same way as the one obtained in \cite{BW}, although the latter dealt only with weakly dependent AR$(\infty)$ processes. \\
\begin{Rem}
As already mentioned, Boubacar Ma\"inassara et al. (2021) \cite{BES} have also established the almost certain convergence and asymptotic normality of the QML estimator in the specific case of FARIMA processes, but allowing the white noise $(\varepsilon_t)$ to be a weak white noise (non-correlation) and not a strong white noise as in our work. This comes at the price of slightly stronger moment conditions: in \cite{BES}, a moment of order $2+\nu$ is required for almost sure convergence and a moment of order $4+\nu$ for asymptotic normality (with $0<\nu<1$). This is the price to pay in their Assumption A4 for working with strong mixing properties of  $(\varepsilon_t)$.\\
\end{Rem}
\begin{Rem}
Of course, in this specific context of linear long-memory processes, we would like to make a comparison between the asymptotic results for the convergence of the QMLE estimator and those obtained with Whittle's estimator in \cite{gi-sur}. In this paper, more precisely in Theorem 4, the asymptotic covariance matrix of $\widehat \gamma_n $ is given by the spectral density $f_\gamma$ and is written as $(4 \pi)^{-1}\, \int_{-\pi}^\pi \big (\partial_\gamma \log (f_\gamma(\lambda)) \big ) \, {}^t \big (\partial_\gamma \log (f_\gamma(\lambda)) \big )\, d\lambda$. However, Dahlhaus in \cite{Da} has shown that this asymptotic covariance matrix is also that of the maximum likelihood estimator in the case of a Gaussian process, the latter also being $(M^*)^{-1}$ if $(\varepsilon_t)$ is Gaussian white noise. This means that asymptotically, the QML and Whittle estimators behave identically. However, we will see a slight numerical advantage due to the convergence of the QMLE in the case of observed trajectories whose size is not too large.\\
\end{Rem}
\subsection{Case of a non-centered long-memory linear process}
Finally, we can consider the special case where the process $(X_t)$ is not centered and estimate the location parameter $\mu^*=\E[X_0]$. This means that $(X_t)$ can now be written as:
\begin{equation}\label{location}
X_t=\mu^* + \sum_{i=0}^\infty a_i(\theta^*)\, \varepsilon_{t-i}\quad \mbox{for all $t\in \Z$,}
\end{equation}
with the same assumptions on $\theta$, on $(a_i(\theta))$ and on $(\varepsilon_t)$. \\
~\\
First of all, the AR$(\infty)$ representation we used to define $(X_t)$ does not allow $\mu^*$ to intervene, so the QML estimator can not estimate this parameter. So, if $(X_t)$ satisfies \eqref{location}, then $(X_t)$ still satisfies \eqref{LRD7}. This is because $\sum_{k=1}^\infty u_k(\theta)=1$ for any $\theta$. Consequently, the QML estimate of the parameter $\theta$ is not at all affected by the fact that $(X_t)$ is not a centered process and verifies \eqref{location} and Theorems \ref{Theo1} and \ref{theo2} are still valid. Note that the same applies to the Whittle's estimator, as it was already remarked in Dahlhaus (1989). \\
~\\
Concerning the estimation of the localization parameter $\E[X_0]=\mu^*$ for long-memory processes, this question has been the subject of numerous publications. Among the most important are \cite{A}, \cite{ST} and the review article \cite{Be93}. We are dealing here with long-memory linear processes, and the article \cite{A} had already shown the most important point: we can not expect a convergence rate in $\sqrt n$, contrary to the other process parameters. In the case of the QML estimator, this can be explained by the fact that $\mu^*$ cannot intervene in the equation \eqref{LRD7}, contrary to what would happen for an ARMA process, for example.\\
~\\
More precisely, from these references, under the assumptions of Theorem \ref{theo2} except that $(X_t)$ is defined by \eqref{location}, we obtain:
$$
\frac {n^{1/2-d(\theta^*)}}{L_a(n)} \, \big ( \overline X_n - \mu^* \big ) \limiteloin {\cal N} \Big (0 \, , \, \frac {C_{d(\theta^*)}}{d(\theta^*)\, (2 d(\theta^*)+1)} \Big )\qquad\mbox{with}\quad \overline X_n=\frac 1 n \, \sum_{k=1}^n X_k,
$$ 
and $C_d=\int_0^\infty (u+u^2)^{d-1}\,du$. However, as we are considering linear processes here, Adenstedt (1974) \cite{A} proved by a Gauss-Markov type theorem that there exists a Best Linear Unbiased Estimator (BLUE) and provided its asymptotic efficiency. By adapting its writing, it will be enough to consider the matrix $\Sigma(\theta)= \big ( r_X(|j-i|) \big )_{1\leq i,j\leq n}$ with $\displaystyle r_X(k)= \cov(X_0,X_k)=\sum_{i=0}^{\infty}{a_i(\theta)\, a_{i+k}(\theta)} $ and define:
$$
\widehat \mu_{BLUE}(\theta^*)=\big ( {}^t \1 \, \Sigma^{-1}(\theta^*) \, \1 \big ) ^{-1}   {}^t \1 \, \Sigma^{-1}(\theta^*) \, X, \qquad \mbox{with}\quad X={}^t(X_1,\ldots,X_n).
$$
Then it is established in \cite{A} that $\displaystyle \frac{\var\big (\overline X_n\big )}{\var\big (\widehat \mu_{BLUE}(\theta^*) \big )}\limiten \frac {\pi \, d(\theta^*)\, (2d(\theta^*)+1)}{B\big (1-d(\theta^*),1-d(\theta^*)\big )\, \sin \big (\pi \, d(\theta^*)\big )}$, where $B(a,b)$ is the usual Beta function, and this limit belongs to $[0.98,1]$ when $0<d(\theta^*)<1/2$. Therefore, since $\widehat \mu_{BLUE}(\theta^*)$ is a linear process, we obtain:
$$
\frac {n^{1/2-d(\theta^*)}}{L_a(n)} \, \big ( \widehat \mu_{BLUE}(\theta^*) - \mu^* \big ) \limiteloin {\cal N} \Big (0 \, , \, \frac {\pi \,C_{d(\theta^*)}}{B\big (1-d(\theta^*),1-d(\theta^*)\big )\, \sin \big (\pi \, d(\theta^*)\big )} \Big ).
$$
Finally, the estimation of $\theta^*$ by $\widehat \theta_n$ makes the use of the BLUE estimator of $\mu^*$ effective. Indeed, as $\widehat \theta_n$ is a convergent estimator of $\theta^*$, as $\theta \in (0,0.5) \mapsto \big ( {}^t \1 \, \Sigma^{-1}(\theta) \, \1 \big ) ^{-1} {}^t \1 \, \Sigma^{-1}(\theta)$ is a continuous function, we deduce by Slutsky's lemma that:
$$
\frac {n^{1/2-d(\theta^*)}}{L_a(n)} \, \big ( \widehat \mu_{BLUE}(\widehat \theta_n) - \mu^* \big ) \limiteloin {\cal N} \Big (0 \, , \, \frac {\pi \,C_{d(\theta^*)}}{B\big (1-d(\theta^*),1-d(\theta^*)\big )\, \sin \big (\pi \, d(\theta^*)\big )} \Big ).
$$
\section{Numerical applications}\label{simu}
\subsection{Numerical simulations}
In this section, we report the results of Monte Carlo experiments conducted with different long-memory causal linear processes. More specifically, we considered:
\begin{itemize}
\item Three different processes generated from Gaussian standard white noises:
\begin{enumerate}
\item A FARIMA$(0,d,0)$ process, denoted {\bf FARIMA}, with parameters $\sigma^2=4$ and $d=0.1$, $0.2$, $0.3$ and $0.4$;
\item A FARIMA$(1,d,0)$ process, denoted {\bf FARIMA(1,d,0)}, with parameters $\sigma^2=4$ and $d=0.1$, $0.2$, $0.3$ and $0.4$, and AR-parameter $\alpha=0.5$ and $0.9$;
\item A long-memory causal affine process, denoted {\bf LM}, defined by:
$$
X_t=a_0 \, \varepsilon_t + \zeta(1+d)^{-1} \, \sum_{k=1}^\infty k^{-1-d} \, X_{t-k}\quad \mbox{for any $t\in \Z$},
$$
with parameters $\sigma^2=4$ and $d=0.1$, $0.2$, $0.3$ and $0.4$.
\end{enumerate}
\item Several trajectory lengths: $n=300,~1000,~3000$ and $10000$.
\item In the case of the {\bf FARIMA} process, we compared the accuracy of the QMLE with the one of the Whittle estimator which also satisfies a central limit theorem (see \cite{gi-sur}). We denote $\widehat \theta_{W}=( \widehat d_W,\widehat \sigma_W^2)$ this estimator. \\
\end{itemize}
\noindent The results are presented in Tables \ref{Table1}, \ref{Table7} and \ref{Table2}. \\
~\\
\begin{table*}
\centering
\begin{tabular}{l|c||c|c||c|c||c|c||c|c|}
n & & \multicolumn{2}{c|}{$d=0.1$, $\sigma^2=4$} &\multicolumn{2}{c|}{$d=0.2$, $\sigma^2=4$} &\multicolumn{2}{c|}{$d=0.3$, $\sigma^2=4$} &\multicolumn{2}{c|}{$d=0.4$, $\sigma^2=4$}  \\
\hline  
 $300$&$\widehat \theta_n=( \widehat d_n,\widehat \sigma_n^2)$ &0.045 & 0.327 &0.045   &0.317
&0.046 &0.318 
 &0.050 &0.327\\
& $\widehat \theta_{W}=(\widehat d_W, \widehat \sigma_W^2)$& 0.050   &0.327 & 0.050 &0.318
& 0.051 & 0.319
 & 0.053 &0.332
\\
$1000$ & $\widehat \theta_n=( \widehat d_n,\widehat \sigma_n^2)$ &0.024 & 0.179&
0.024 & 0.179  
& 0.025 & 0.183&
0.025 & 0.184  \\
& $\widehat \theta_{W}=(\widehat d_W, \widehat \sigma_W^2)$ & 0.026  & 0.179&
 0.026 & 0.179 
& 0.026 & 0.183
 &0.026 &0.185 \\
$3000$&$\widehat \theta_n=( \widehat d_n,\widehat \sigma_n^2)$ & 0.014 &0.103  &0.014 &
0.105 & 0.014 & 
0.103 & 0.015 & 
0.100 \\
& $\widehat \theta_{W}=(\widehat d_W, \widehat \sigma_W^2)$ & 0.014 &0.103  &0.015 &
 0.106& 0.014 &
0.103 & 0.015
&0.100 \\
$10000$&$\widehat \theta_n=( \widehat d_n,\widehat \sigma_n^2)$ & 0.007 & 0.056 & 0.007 & 0.056 & 0.008 & 0.056 & 0.008   & 0.052   \\
& $\widehat \theta_{W}=(\widehat d_W, \widehat \sigma_W^2)$  & 0.007 & 0.056 & 0.007 & 0.056 &  0.008 &0.057   & 0.008  & 0.052\\
\hline
\end{tabular}
\caption{Square roots of the MSE computed for the QMLE $\widehat \theta_n$ and the Whittle estimator in the case of a {\bf FARIMA} process computed from $1000$ independent replications. }
\label{Table1}
\end{table*}
\begin{center}
\begin{figure}[h]
\includegraphics[scale=0.35]{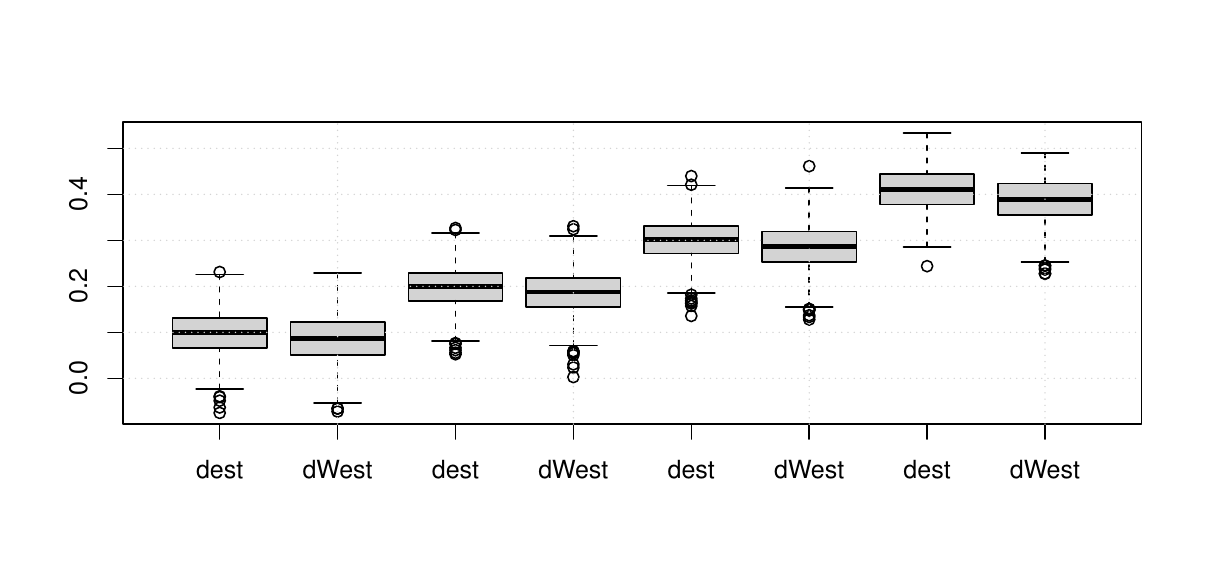}\includegraphics[scale=0.33]{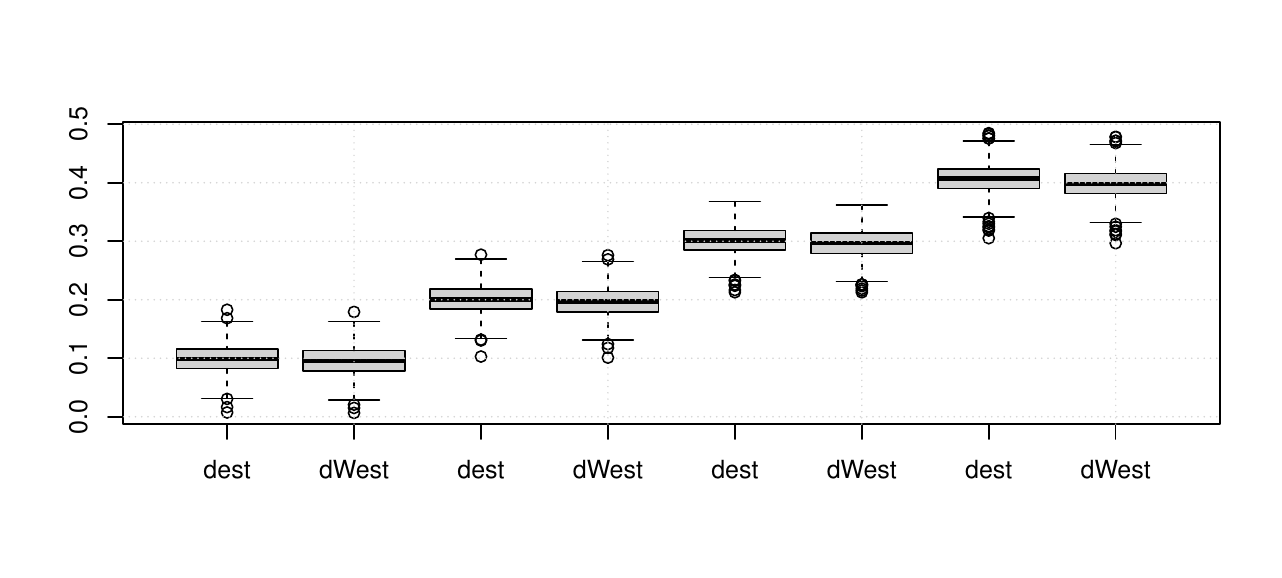}
\includegraphics[scale=0.32]{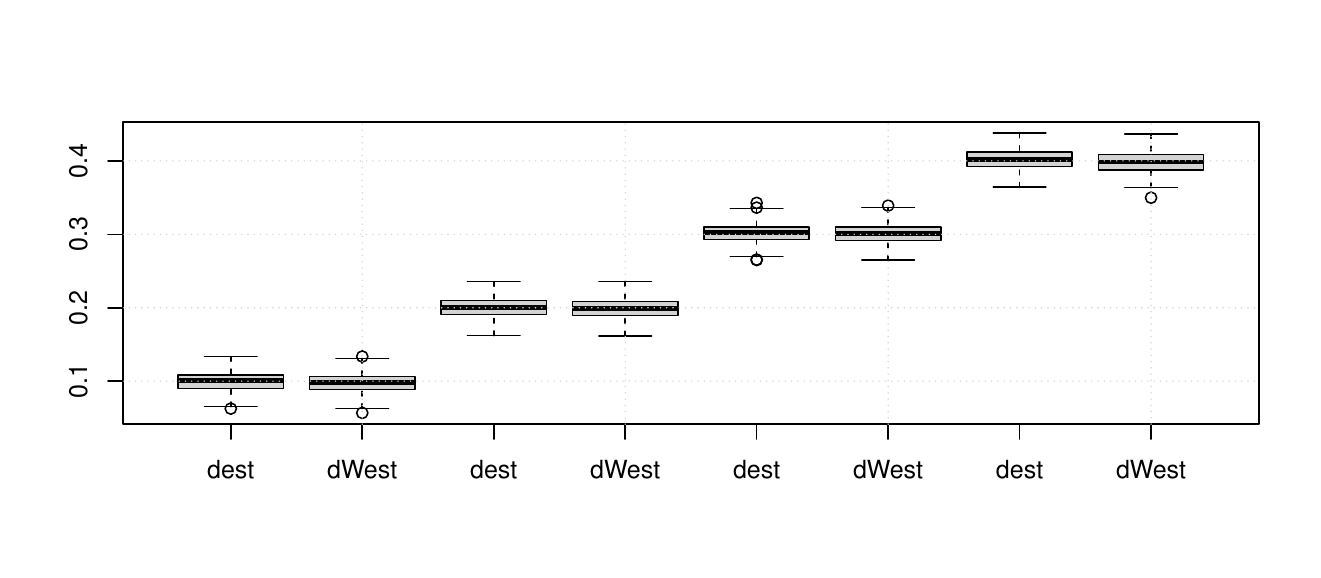}\includegraphics[scale=0.32]{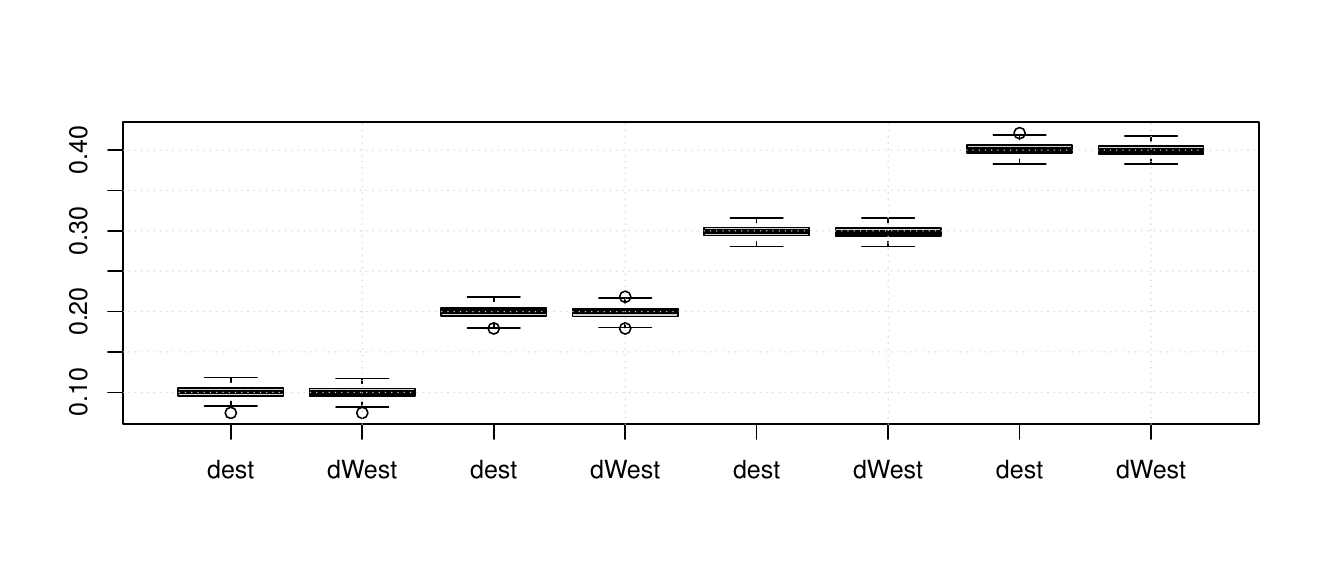}
\caption{Boxplots for estimating $d$ on FARIMA$(0,d,0)$ processes with the estimators $\widehat d_n$ and $\widehat d_W$ (denoted {\tt dest} and {\tt dWest} for $n=300$ (top left), $n=1000$ (top right), $n=3000$ (bottom left) and $n=10000$ (bottom right).} \label{Fig1}
\end{figure}
\end{center}
~\\
~\\
\begin{table*}
\centering
\begin{tabular}{ll||c|c||c|c||c|c||c|c|}
& &  \multicolumn{2}{c|}{$n=300$} &\multicolumn{2}{c|}{$n=1000$} &\multicolumn{2}{c|}{$n=3000$} &\multicolumn{2}{c|}{$n=10000$}  \\
& &  $ \widehat d_n$ & $\widehat \alpha_n$ & $ \widehat d_n$ & $\widehat \alpha_n$  & $ \widehat d_n$ & $\widehat \alpha_n$  & $ \widehat d_n$ & $\widehat \alpha_n$  \\
\hline  
 $d=0.1$ & $\alpha=0.5$ & 0.099 & 0.116 & 0.063  & 0.070 &0.043 & 0.048
& 0.024 & 0.026 \\
&  $\alpha=0.9$ &  0.129 & 0.110 &
 0.060 & 0.041 & 
0.029 & 0.017 & 0.015  & 0.008 \\
\hline
 $d=0.2$ & $\alpha=0.5$ &  0.128 & 0.140 & 0.080 & 0.084 & 0.041 & 0.046
& 0.022 & 0.023\\
&  $\alpha=0.9$ &  0.110& 0.084 &
 0.052 &0.030& 
 0.028& 0.016&  0.014 & 0.009\\
\hline
 $d=0.3$ & $\alpha=0.5$ &  0.148 & 0.158 & 0.081 & 0.088 & 0.042  & 0.045
& 0.023 & 0.026 \\
&  $\alpha=0.9$ & 0.103 & 0.067 &
  0.055 & 0.034& 
0.027 & 0.015&  0.016 & 0.009 \\
\hline
 $d=0.4$ & $\alpha=0.5$ &   0.197 & 0.202 &  0.119 & 0.125 &  0.042 & 0.045 
& 0.023 & 0.025\\
&  $\alpha=0.9$ &0.132  & 0.054&
 0.069 & 0.034 & 
 0.036& 0.017 & 0.018 & 0.009 \\
\hline
\end{tabular}
\caption{Square roots of the MSE computed for the QMLE $\widehat \theta_n$ in the case of the {\bf FARIMA(1,d,0)} process computed from $1000$ independent replications. }
\label{Table7}
\end{table*}
~\\
~\\
\begin{table*}
\centering
\begin{tabular}{l|c||c|c||c|c||c|c||c|c|}
n & & \multicolumn{2}{c|}{$d=0.1$, $\sigma^2=4$} &\multicolumn{2}{c|}{$d=0.2$, $\sigma^2=4$} &\multicolumn{2}{c|}{$d=0.3$, $\sigma^2=4$} &\multicolumn{2}{c|}{$d=0.4$, $\sigma^2=4$}  \\
\hline  
 $300$&$\widehat \theta_n=( \widehat d_n,\widehat \sigma_n^2)$ & 0.048 & 0.082 & 0.054
& 0.083 & 0.059 & 0.080 
& 0.065 &0.080 \\
$1000$ & $\widehat \theta_n=( \widehat d_n,\widehat \sigma_n^2)$& 0.025& 0.045 &
 0.032& 0.047 & 
 0.032 & 0.045&  0.038& 
0.046 \\
$3000$&$\widehat \theta_n=( \widehat d_n,\widehat \sigma_n^2)$  &0.014 & 0.025 &0.017 &
  0.027 &0.018 &
0.024 & 0.020  & 
0.026\\
$10000$&$\widehat \theta_n=( \widehat d_n,\widehat \sigma_n^2)$ & 0.008 & 0.013 &0.010 & 0.013 & 0.011 &  0.015  &0.012 & 0.014 \\
\hline
\end{tabular}
\caption{Square roots of the MSE computed for the QMLE $\widehat \theta_n$ in the case of the {\bf LM} process computed from $1000$ independent replications. }
\label{Table2}
\end{table*}
~\\
\noindent The results of Tables \ref{Table1} and \ref{Table2} show a weak effect of the value of $d$ on the speed of convergence of the $\widehat d_n$ estimator and, more generally, of $\widehat \theta_n$, which may seem counter-intuitive since the long memory being stronger, the effect of initial values should be stronger. To investigate this further, we carried out new numerical studies using simulations of the {\bf FARIMA} process for values of $d$ approaching $0.5$, {\it i.e.} $d=0.43$, $d=0.46$ and $d=0.49$, and the results are shown in Table \ref{Table043}. \\
\begin{table*}
\centering
\begin{tabular}{l|c||c|c||c|c||c|c|}
n & & \multicolumn{2}{c|}{$d=0.43$, $\sigma^2=4$} &\multicolumn{2}{c|}{$d=0.46$, $\sigma^2=4$} &\multicolumn{2}{c|}{$d=0.49$, $\sigma^2=4$}   \\
\hline  
 $300$&$\widehat \theta_n=( \widehat d_n,\widehat \sigma_n^2)$ & 0.053 & 0.328 & 0.065 & 0.369  & 0.113 & 0.633 
\\
$1000$ & $\widehat \theta_n=( \widehat d_n,\widehat \sigma_n^2)$& 0.028 & 0.177 &  0.036 & 
 0.189 &  0.066 & 
 0.283   \\
$3000$&$\widehat \theta_n=( \widehat d_n,\widehat \sigma_n^2)$  &0.016 & 0.113 & 0.018& 0.109 &0.036 & 0.132\\
$10000$&$\widehat \theta_n=( \widehat d_n,\widehat \sigma_n^2)$ & 0.009 & 0.059& 0.011& 0.057 & 0.021 & 0.071  \\
\hline
\end{tabular}
\caption{Square roots of the MSE computed for the QMLE $\widehat \theta_n$ of {\bf FARIMA} process computed from $1000$ independent replications when $d$ is close to $0.5$. }\label{Table043}
\end{table*}
~\\
\noindent {\bf Conclusions of the simulations:} 
\begin{enumerate}
\item The results of the simulations show that the consistency of the QML estimator $\widehat \theta_n$ is satisfied and also that its $1/\sqrt n$ convergence rate of the estimators almost occurs for all processes considered.
\item The value of the parameter $d$ seems to have little influence on the speed of convergence of the estimators as long as $d$ does not get too close to $0.5$. However, when we consider values of $d$ which increase towards $0.5$, if the rate of convergence still looks good in $\sqrt n$, the asymptotic variance considerably increases.
\item When a short-memory component is added to the long-memory component, as in the case of a FARIMA$(1,d,0)$ process, the rate of convergence to $d$ deteriorates, especially for small trajectories. But the rate of convergence still seems to be in $\sqrt n$. We can also see that the rate of convergence deteriorates much more sharply than for the FARIMA$(0,d,0)$ process as $d$ increases towards $0.5$.
\item In the case of the FARIMA process, the comparison between the QML and Whittle estimators leads to very similar results for large $n$, but for $n=300$ the QML estimator provides slightly more accurate estimate, in particular with a more centered distribution around the estimated value.
\end{enumerate}
\subsection{Application on real data}
Here, we will apply the QML estimator to a time series observation known to have a long memory. These are monthly temperature (in degree Celsius) for the northern hemisphere for the years 1854-1989, from the data base held at the Climate Research Unit of the University of East Anglia, Norwich, England. The numbers consist of the temperature difference from the monthly average over the period 1950-1979. For our purposes, and given the general rise in temperatures due to climate change, it is preferable to work on detrended data, for example using simple linear regression, as had already been done in \cite{Be}. Figure 2 shows the two time series:
\begin{center}
\begin{figure}
\includegraphics[scale=0.35]{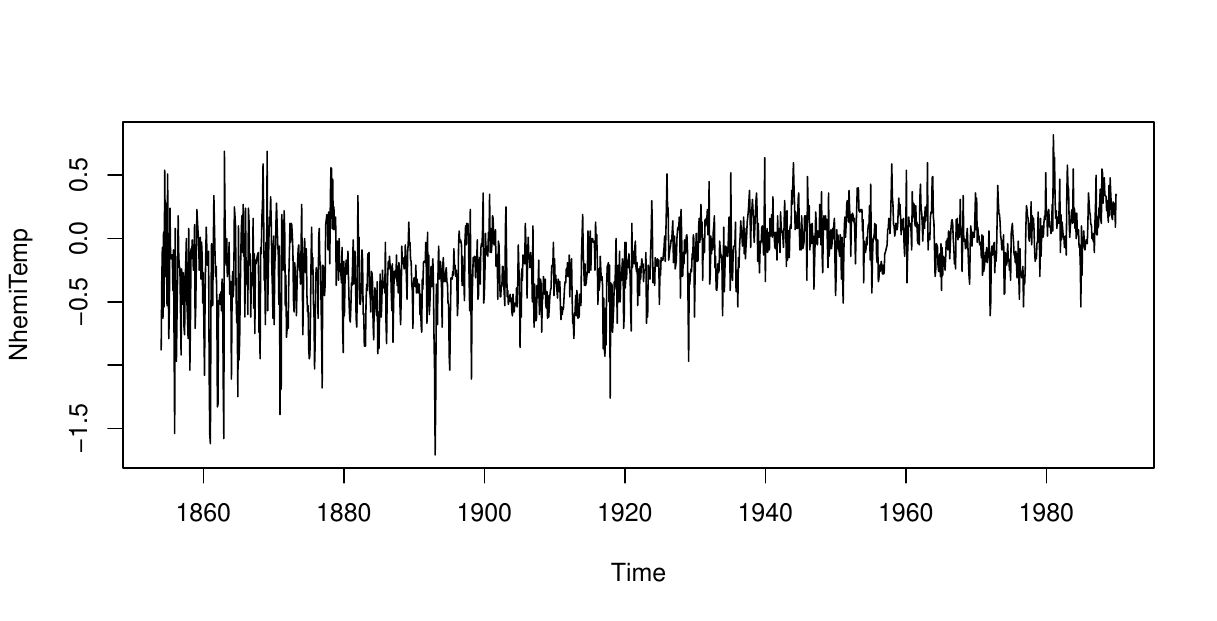}\includegraphics[scale=0.35]{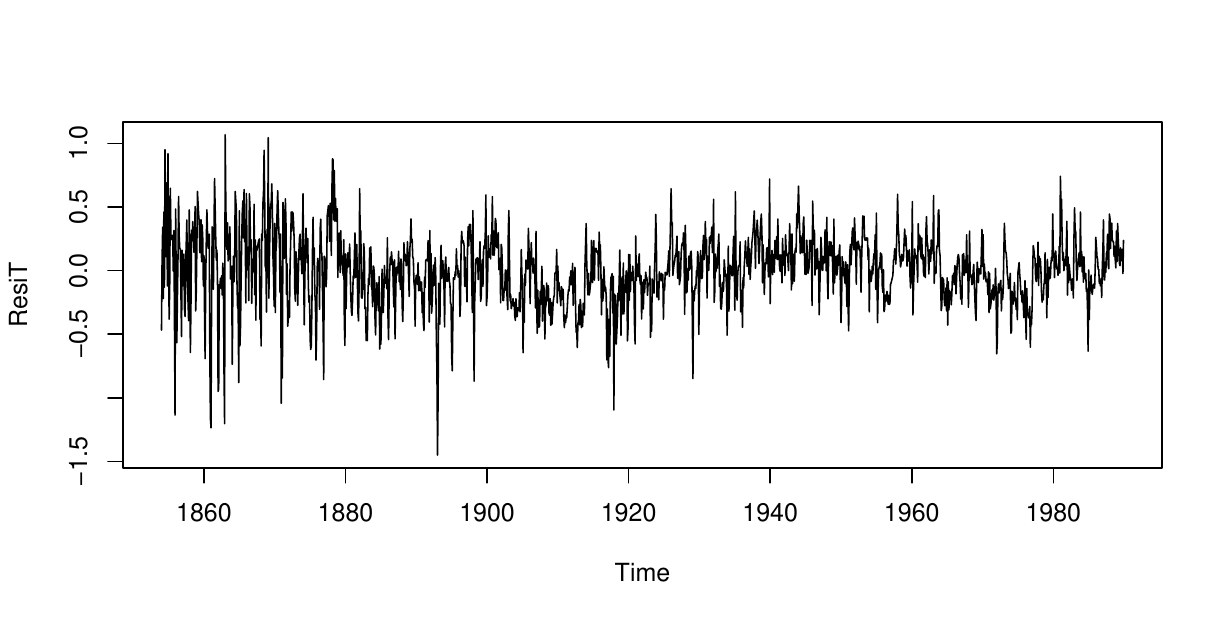}
\caption{Recentered series of des monthly temperatures (in degree Celsius) for the northern hemisphere for the years 1854-1989 (left) and the same series detrended by simple linear regression (right).} \label{Fig2}
\end{figure}
\end{center}
These data have been studied in \cite{Be} (see for example p.179), and Whittle's estimator of the long memory parameter for a FARIMA process applied to detrended data yielded $\widehat d_W \simeq 0.37$, while the observed path size is $n=1632$.\\
We applied the QML estimator for the FARIMA$(0,d,0)$ process to this same series and, as we might have expected, the result was almost identical $\widehat d_n \simeq 0.37$, with $\widehat \sigma_n \simeq 0.056$. We also applied the QML estimator for processes {\bf LM} and the result obtained is rather $\widehat d_n \simeq 0.44$, which is not very far from the previous value. This confirms the long-memory nature of this series, and the implementation of a goodness-of-fit test could enable us to go a little further in choosing between the 2 models or others (note that such a test has been implemented for FARIMA processes in \cite{BES23}).
\section{Conclusion}
In this paper, we have shown that the QML estimator, which offers excellent convergence results for parameters of classical short-memory time series such as GARCH, ARMA, ARMA-GARCH or APARCH processes, also gives excellent results for long-memory time series. This had already been established for FARIMA processes, even with weak white noise, in \cite{BES}. And we generalize this to all long-memory linear processes, offering a very interesting alternative to Whittle estimation, both from a theoretical and a numerical point of view.
\section{Proofs}\label{proofs}
\subsection{Proofs of the main results}
\begin{proof}[Proof of Proposition \ref{prop1}]
Using $B$ the lag or backshift linear operator on $\R^{\Z}$, we can denote $X=S(B)\, \varepsilon$, where $X=(X_t)_{t\in \Z}$ and $\varepsilon=(\varepsilon_t)_{t\in \Z}$ and $S(B)=\sum_{i=0}^\infty a_i \, B^{i}$. We know that there exists a linear operator denoted $S^{-1}$ such as  $\varepsilon=S^{-1}(B)\, X$. As a consequence, $X=a_0\, \varepsilon+(S(B)-a_0\, I_d)\,\varepsilon=a_0\, \varepsilon+(S(B)-a_0\, I_d)S^{-1}(B)\, X=a_0\, \varepsilon+(I_d-a_0\, S^{-1}(B))\, X$ which is the affine causal representation of $X$. \\
Let $X_t=a_0\, \varepsilon_t+\sum_{i=1}^\infty u_i \, X_{t-i}$. Then, for any $t\in \Z$,
\begin{eqnarray*}
X_t&=&a_0\, \varepsilon_t+\sum_{i=1}^\infty u_i\, X_{t-i} \\ 
&=& a_0\, \varepsilon_t+ \sum_{i=1}^\infty \sum_{j=0}^\infty u_i\,a_j\, \varepsilon_{t-i-j} \\
& =& a_0\, \varepsilon_t+  \sum_{k=1}^\infty \Big ( \sum_{j=0}^{k-1} u_{k-j}\, a_j  \Big ) \, \varepsilon_{t-k}. 
\end{eqnarray*}
As a consequence, denoting $u_0=-1$, for any $k \in \N^*$, 
\begin{equation}\label{prod}
 \sum_{j=0}^{k-1} u_{k-j}\, a_j =a_k \quad
\Longrightarrow\quad \Big (\sum_{i=0}^k u_i \Big ) \, \Big (\sum_{j=0}^k a_j \Big )=0.
\end{equation}
Finally, since the convergence radius of the power series $\sum_{\ell=0}^\infty  a_\ell\, z^\ell$ is $1$ from asymptotic expansion \eqref{a2}, we deduce that for any $z\in \C$, $|z|<1$, 
\begin{equation}\label{prod2}
\Big (\sum_{k=0}^\infty u_k \, z^k \Big ) \, \Big (\sum_{\ell=0}^\infty  a_\ell\, z^\ell \Big )=-a_0.
\end{equation}
Now, we are going to use a Karamata Tauberian theorem as it is stated in Corollary 1.7.3 of Bingham {\it et al.} (1987):\\
~\\
{\em 
Fix $\rho>0$ and let $L$ a slow varying function. Then if $(\alpha_n)_{n \in \N}$  is a sequence of nonnegative real numbers and the power series $A(s)=\sum_{n=0}^\infty \alpha_n\, s^n$ converges for any $s\in[0,1)$, then 
\begin{equation}\label{Kara}
\sum_{k=0}^n\alpha_k \equin L(n)\, n^\rho ~\Longleftrightarrow ~
A(s)\sim \frac{\Gamma(1+\rho)}{ (1-s)^{\rho}}\, L\big ((1-s)^{-1}\big )~\mbox{as $s\to 1^{-}$.}
\end{equation}}
\newline
\noindent Note that this result is also established if there exists $N_0\in \N$ such as $(\alpha_n)_{n \geq N_0}$  is a sequence of nonnegative real numbers. We first apply \eqref{Kara} to $(\alpha_n)=(a_n)$. Indeed, from \eqref{a2} and with $\rho=d$, there exists $N_0\in \N$ such as $(a_n)_{n \geq N_0}$  is a sequence of nonnegative real numbers and $\sum_{k=0}^n a_k \equin L(n)\, n^\rho$ with $L(\cdot)=\frac {L_a(\cdot)}{d}$.  Therefore, we deduce that
\begin{equation}\label{Kara2}
\sum_{n=0}^\infty a_n\, s^n\sim \frac{\Gamma(1+d)}{d\, (1-s)^{d}}\, L_a\big ((1-s)^{-1}\big )~\mbox{as $s\to 1^{-}$.}
\end{equation}
Therefore, from \eqref{prod2}, the following expansion can be deduced:
\begin{equation}\label{Kara3}
\sum_{n=0}^\infty u_n\, s^n\sim -\frac{a_0\, (1-s)^{d}}{\Gamma(d)}\, L^{-1}_a\big ((1-s)^{-1}\big )~\mbox{as $s\to 1^{-}$.}
\end{equation}
On the other hand, if we consider \eqref{prod2} when $s\to 1^{-}$, $\sum_{\ell=0}^\infty  a_\ell\, s^\ell \to \infty$ since $(a_n)$ satisfies \eqref{a2}. As a consequence,
$$
\sum_{n=0}^\infty u_n\, s^n \to 0=\sum_{n=0}^\infty u_n \quad \mbox{when $s\to 1^{-}$}.
$$
We deduce that $u_n \limiten 0$ and the sequence $(U_n)_{n\in \N}$ can be defined where we denote $U_n=\sum_{k=n+1}^\infty u_k$. But since $\sum_{n=0}^\infty u_n=0$, for any $s\in [0,1]$,
$$
\sum_{k=0}^\infty u_k\, s^k=(s-1)\, \sum_{k=0}^\infty U_k\, s^k.
$$ 
Using \eqref{Kara3}, we deduce
$$
\sum_{k=0}^\infty U_k\, s^k \sim \frac{a_0\, (1-s)^{d-1}}{\Gamma(d)}\, L^{-1}_a\big ((1-s)^{-1}\big )~\mbox{as $s\to 1^{-}$.}
$$
From \eqref{prod}, we also have for any $n \in \N$
\begin{equation}\label{prod3}
\Big (\sum_{k=0}^n u_k \Big ) \, \Big (\sum_{\ell=0}^n  a_\ell \Big )=-a_0
\end{equation}
Since $(a_n)$ satisfies \eqref{a2}, we know that there exists $N_0$ such as $a_n>0$ and $\sum_{\ell=0}^n  a_\ell >0$ for any $n \geq N_0$. Therefore we know from \eqref{prod3} that for any $n \geq N_0$, $\sum_{k=0}^n u_k<0$ and thus $U_n>0$ since $\sum_{k=0}^\infty u_k=0$. Thus we can apply \eqref{Kara} to $(\alpha_n)=(U_n)$ with $\rho=1-d$ and this induces
$$
\sum_{k=0}^n U_k \equin \frac{a_0\,}{\Gamma(d)\, \Gamma(2-d)}\, L^{-1}_a (n )\, n^{1-d}.
$$
Since for $n\geq N_0$, $u_n>0$, we deduce that $(U_n)$ is a positive decreasing sequence for $n\geq N_0$. Using again Bingham {\it et al.} (1987), we deduce that
$$
U_n \equin \frac{a_0\,(1-d)}{\Gamma(d)\, \Gamma(2-d)}\, L^{-1}_a\big (n\big )\, n^{-d}=\frac{a_0}{\Gamma(d)\, \Gamma(1-d)}\, L^{-1}_a\big (n\big )\, n^{-d}.
$$
To finish with, since $(U_n)$ is a positive decreasing sequence for $n \geq N_0$, we deduce:
$$
u_n=\frac{a_0\,d }{\Gamma(d)\, \Gamma(1-d)}\, L^{-1}_a\big (n\big )\,n^{-1-d},
$$
and this achieves the proof.
\end{proof}

\begin{proof}[Proof of Theorem \ref{Theo1}]
In the sequel, we will denote for any $t\in \N^*$ and $\theta \in \Theta$,
\begin{equation}\label{mtt}
\widetilde m_t(\theta)=m_t(\theta)-\widehat m_t(\theta)=\sum_{k=t}^{\infty} u_k(\theta)\, X_{t-k}.
\end{equation}
For a random variable $Z$ and $r\geq 1$, denote $\|Z\|_r=\big (\E \big [ |Z|^r \big ]\big )^{1/r}$.\\
~\\
\noindent {\bf 1.} Firstly we prove some useful inequalities. \\
~\\
From the Cauchy-Schwarz Inequality, for any $\theta \in \Theta$ and $t\in \Z$,
\begin{eqnarray*}
\big (m_t(\theta)\big )^2 &\leq& \Big (\sum_{k=1}^\infty \big |u_k(\theta) \big |\Big )\, \Big (   
\sum_{k=1}^\infty \big |u_k(\theta) \big |\, X_{t-k}^2 \Big ) \\
&\leq& \sup_{\theta \in \Theta} \Big \{ \sum_{k=1}^\infty \big |u_k(\theta) \big |\Big \} \,    \sup_{\theta \in \Theta} \Big \{
\sum_{k=1}^\infty \big |u_k(\theta) \big |\, X_{t-k}^2 \Big \} \\
\Longrightarrow ~\big \|\sup_{\theta \in \Theta}  \big |m_t(\theta)\big | \big \|_2^2 & \leq & 
\Big (\sup_{\theta \in \Theta} \Big \{ \sum_{k=1}^\infty \big |u_k(\theta) \big |\Big \}\Big )^2\, \big \| X_0\big \|_2^2<\infty,
\end{eqnarray*}
since $(u_k)$ follows \eqref{untheta}, $\Theta$ is a compact subset, $\theta \in \Theta \mapsto u_k(\theta)$ is a continuous function for any $k\geq 1$ and $d(\theta)\in (0,1/2)$. \\
Using the same inequalities we also obtain that there exists $C_2>0$ such that for any $t\geq 1$,
\begin{multline}\label{boundm}
\big \|\sup_{\theta \in \Theta} \big | \widehat m_t(\theta)\big | \big \|_2^2<\infty \quad \\
\mbox{and}\quad \big \|\sup_{\theta \in \Theta} \big | \widetilde m_t(\theta)\big | \big \|_2^2 \leq \Big (\sup_{\theta \in \Theta} \Big \{ \sum_{k=t+1}^\infty \big |u_k(\theta) \big |\Big \}\Big )^2\, \big \| X_0\big \|_2^2 \leq C_2 \, t^{-2\, \underline d},
\end{multline}
with $0 < \underline d < \inf_{\theta \in \Theta} d(\theta)$ from the condition \eqref{untheta} on $(u_n(\theta))$. \\ 

Finally with 
\begin{equation}\label{q}
\left \{ \begin{array}{ccc} q_t(\theta)&=&-\frac{1}{2}\Big (\log \big(\sigma^2\big ) + \frac{\big (X_{t}- m_t(\theta)\big)^{2}}{ \sigma^2}\Big )  \\
 \widehat q_t(\theta)&=&-\frac{1}{2}\Big (\log \big(\widehat \sigma^2\big ) + \frac{\big (X_{t}-\widehat m_t(\theta)\big)^{2}}{\sigma^2}\Big  ) \end{array} \right .,
\end{equation} 
we obtain from the previous bounds and $\sigma^2 \in [\sigma^2_m,\sigma^2_M]$ where $0<\sigma^2_m< \sigma^2_M$,
\begin{eqnarray}
\nonumber \sup_{\theta \in \Theta} \big |q_t(\theta) \big | & \leq &\sup_{\theta \in \Theta} \Big \{  \frac 1 { \sigma_m^2} \big ( X_t^2+m^2_t(\theta)\big )+\frac 1 2 \, \big | \log(\sigma_M^2)\big | \Big \}\\
\nonumber  \Longrightarrow ~~ \Big \| \sup _{\theta \in \Theta} \big |q_t(\theta) \big | \Big \|_1 & \leq & \frac 1 {\sigma_m^2} \big ( \| X_t\|_2^2+ \big \| \sup _{\theta \in \Theta} \big |m_t(\theta)\big |\big \|_2^2 \big ) +\frac 1 2 \, \big | \log(\sigma_M^2)\big | \\
& < & \infty. \label{inegq}
\end{eqnarray}
And to conclude with these preliminary bounds, using Cauchy-Schwarz and the triangular inequality, there exists $C>0$ such as for  $t\geq 1$,
\begin{eqnarray} 
\nonumber \Big \| \sup _{\theta \in \Theta} \big |q_t(\theta) - \widehat q_t(\theta)\big | \Big \|_1 &\leq &\frac 1 2 \, \Big \| \sup _{\theta \in \Theta} \big |2X_t +m_t(\theta)+ \widehat m_t(\theta)\big | \Big \|_2 \, \Big \| \sup _{\theta \in \Theta} \big | \widetilde m_t(\theta)\big | \Big \|_2 \\
\nonumber & & \hspace{-2cm}\leq \frac 1 2 \, \Big ( 2 \, \|X_0^2\|_2^2+ \big \|  \sup _{\theta \in \Theta} \big |m_t(\theta)\big | \big \|_2^2+ \big \|  \sup _{\theta \in \Theta} \big |\widehat m_t(\theta)\big | \big \|_2^2\Big ) \,\big ( C_2 \, t^{-2\, \underline d} \big )^{1/2} \\
& & \hspace{-2cm} \leq  C \,\,t^{- \underline d}.\label{inegqq}
\end{eqnarray}
\noindent {\bf 2.} From its AR$(\infty$) representation \eqref{a2}, and since $\|X_0\|_2<\infty$, then $(X_t)_{t\in \Z}$ is a second order ergodic stationary sequence (see Theorem 36.4 in Billingsley, 1995). But for any $\theta \in \Theta$, there exists $H^q_\theta:\R^\N \to \R$ such that
$$
q_t(\theta)=H^q_\theta \big ( (\varepsilon_{t-j})_{j\geq 0}\big ),
$$
with also $\E \big [\big | q_t(\theta) \big |\big ]<\infty$ from \eqref{inegq}. Then using Theorem 36.4 in Billingsley (1995),  $\big (q_t(\theta) \big )\big )_{t\in \Z}$ is an ergodic stationary sequence for any $\theta\in \Theta$ and therefore
$$
 I_n(\theta)\limiteasn \E \big [ q_0(\theta)\big ] \quad \mbox{for any $\theta \in \Theta$},
$$
with $I_n(\theta)$ defined in \eqref{I}. Moreover, since $\Theta$ is a compact set and since we have $\E \big [\sup_{\theta \in \Theta} \big | q_t(\theta) \big |\big ]<\infty$ from \eqref{inegq}, using Theorem 2.2.1. in Straumann (2005), we deduce that $\big (q_t(\theta) \big )\big )_{t\in \Z}$ also follows a uniform ergodic theorem  and we obtain 
\begin{equation}\label{convI}
\sup_{\theta \in \Theta} \big |  I_n(\theta)-\E \big [ q_0(\theta)\big ] \big | \limiteasn 0.
\end{equation}
~\\
Now, using $\widehat I_n(\theta)$ defined in \eqref{QLik}, we can write
\begin{equation}\label{II}
\sup_{\theta \in \Theta} \big | I_n(\theta)-\widehat  I_n(\theta) \big | \leq \frac 1 n \, \sum_{t=1}^n \sup_{\theta \in \Theta} \big | q_t(\theta)-\widehat  q_t(\theta) \big |.
\end{equation}
In Corollary 1 of Kounias and Weng (1969), it is established that for a $\L^1$ sequence of r.v. $(Z_t)_t$ and a sequence of positive real numbers $(b_n)_{n\in \N^*}$ such as $b_n\limiten \infty$, then $\sum_{t=1}^\infty \frac{\E\big [ |Z_t| \big ]}{b_t} <\infty$ implies $\frac 1 {b_n} \, \sum_{t=1}^n Z_t \limiteasn 0$. \\
Therefore, with $b_t=t$ and $Z_t=\sup_{\theta \in \Theta} \big | q_t(\theta)-\widehat  q_t(\theta) \big |$  for $t\in \N^*$, using the inequality  \eqref{inegqq}, 
\begin{multline*}
\sum_{t=1}^\infty \frac 1 t \, \E\big [\sup_{\theta \in \Theta} \big | q_t(\theta)-\widehat  q_t(\theta) \big |\big ] \leq C \, \sum_{t=1}^\infty t^{-\underline d -1} <\infty \quad \\
\Longrightarrow \quad \frac 1 n \, \sum_{t=1}^n \sup_{\theta \in \Theta} \big | q_t(\theta)-\widehat  q_t(\theta) \big | \limiteasn 0.
\end{multline*}
Then, using \eqref{II} and \eqref{convI}, we deduce:
\begin{equation}\label{convII}
\sup_{\theta \in \Theta} \big |  \widehat I_n(\theta)-\E \big [ q_0(\theta)\big ] \big | \limiteasn 0.
\end{equation} 
~\\
{\bf 3.} Finally, the same argument already detailed in the proof of Theorem 1 of Bardet and Wintenberger (2009) is used: $\theta \in \Theta \mapsto \E \big [ q_0(\theta)\big ]$ has a unique maximum reached in $\theta=\theta^*\in \Theta$ because it is assumed that if $u_n(\theta)=u_n(\theta')$ for all $n\in \N^*$ with $\theta=(\gamma,\sigma^2)$ and $\theta'=(\gamma',\sigma^{2})$, then $\theta=\theta'$. This property and the uniform almost sure consistency  \eqref{convII} lead to $\widehat \theta_n \limiteasn \theta^*$. 
\end{proof}

\begin{proof}[Proof of Theorem \ref{theo2}]
As a preamble to this proof, since $\widehat \theta_n \limiteasn \theta^*$ by Theorem \ref{Theo1}, we will be able to reduce the $\Theta$ domain. Let $\widetilde \Theta \subset \Theta$ be a compact set of $\R^p$ such that:
$$ 
\widetilde \Theta=\big\{\theta \in \Theta,~ 2d(\theta^*)-1/2<\inf_{\theta \in \widetilde \Theta} d(\theta)<d(\theta^*) \big \}.
$$
Note that $2d(\theta^*)-1/2<d(\theta^*)$, so it's still possible to determine $\widetilde \Theta$. \\
In the spirit of \eqref{QMLE}, let's define 
$$
\widetilde \theta_n=\text{Arg}\!\max_{\!\!\!\!\!\theta \in \widetilde \Theta} \widehat I_n(\theta).
$$
Using Theorem \ref{Theo1}, it is clear that $\widetilde \theta_n \limiteasn \theta^*$. Moreover, for all $x=(x_1,\ldots,x_p) \in \R^p$, 
\begin{eqnarray*}
\P\Big ( \sqrt n \, \big (\widehat \theta_n -\theta^* \big ) \cart_{j=1}^p (-\infty,x_j]\Big )&& \\
&&\hspace{-3cm}=\P\Big ( \sqrt n \, \big (\widehat \theta_n -\theta^* \big ) \in \cart_{j=1}^p (-\infty,x_j]~\big | ~\widehat \theta_n \in \widetilde \Theta\Big )\, \P\big ( \widehat \theta_n \in \widetilde \Theta\big ) \\
&&\hspace{-2.5cm}+ \P\Big ( \sqrt n \, \big (\widehat \theta_n -\theta^* \big )\in  \cart_{j=1}^p (-\infty, x_j]~\big | ~\widehat \theta_n \notin \widetilde \Theta\Big ) \, \P\big ( \widehat \theta_n \notin \widetilde \Theta\big ) \\
&&\hspace{-3cm}=\P\Big ( \sqrt n \, \big (\widetilde\theta_n -\theta^* \big ) \in \cart_{j=1}^p (-\infty,x_j]\Big )\, \P\big ( \widehat \theta_n \in \widetilde \Theta\big ) \\
&&\hspace{-2.5cm}+  \P\Big ( \sqrt n \, \big (\widetilde \theta_n -\theta^* \big ) \in \cart_{j=1}^p (-\infty, x_j]\Big ) \, \P\big ( \widehat \theta_n \notin \widetilde \Theta\big )
\end{eqnarray*}
Since $\widehat \theta_n \limiteasn \theta^*$ by Theorem \ref{Theo1} and therefore $\P\big ( \widehat \theta_n \notin \widetilde \Theta\big )\limiten 0$ because $\theta^* \in \widetilde \Theta$, it is clear that the asymptotic distribution of $\sqrt n \, \big (\widehat \theta_n -\theta^* \big )$ is the same as the one of $\sqrt n \, \big (\widetilde \theta_n -\theta^* \big )$. Consequently, throughout the rest of the proof, $\Theta$ will be replaced by $\widetilde \Theta$ and $\widehat \theta_n$ by $\widetilde \theta_n$. \\
~\\
In the sequel, for $\theta\in \widetilde \Theta$, we will denote $d=d(\theta)-\varepsilon$ and $d^*_+=d^*+\varepsilon$  where $d^*=d(\theta^*)$ is the unknown long-memory parameter, and we chose $\varepsilon>0$ such as $\varepsilon\leq \frac 1 6 \, \big ( 1-4d(\theta)+2d(\theta^*) \big)$. Hence, from the definition of $\widetilde \Theta$, $1-4d(\theta)+2d(\theta^*)>0$ and 
\begin{equation}\label{dd}
4 \,d^*_+-2\, d-1<0.
\end{equation} 
From Assumption {\bf (A)}, for any $\theta \in \widetilde \Theta$ and $t\in \Z$, $\partial_\theta m_t(\theta)$ and $\partial^2_\theta m_t(\theta)$ a.s. exist with 
$$
\partial_\theta m_t(\theta)=\sum_{k=1}^\infty \partial_\theta u_k(\theta) \, X_{t-k} ~~\mbox{and}~~ \partial^2_{\theta^2} m_t(\theta)=\sum_{k=1}^\infty \partial^2_{\theta^2} u_k(\theta) \, X_{t-k}.
$$
And the same for $\partial_\theta \widehat m_t(\theta)$, $\partial_\theta \widetilde m_t(\theta)$, $\partial^2_\theta \widehat m_t(\theta)$ and $\partial^2_\theta \widetilde m_t(\theta)$. However, note that for any $\theta \in \widetilde \Theta$, $(m_t(\theta))_t$, $(\partial_\theta m_t(\theta))_t$ and $(\partial^2_{\theta^2} m_t(\theta))_t$ are stationary processes while $(\widehat m_t(\theta))_t$, $(\widetilde m_t(\theta))_t$ and their derivatives are not.\\

Due to these results, for any $\theta \in \widetilde \Theta$:
\begin{eqnarray} \label{dqt}
\partial_\theta q_t(\theta)=
\Big ( \begin{array}{c}
\partial_\gamma q_t(\theta) \\
\partial_{\sigma^2} q_t(\theta) 
\end{array} \Big )
=\left ( \begin{array}{c} \frac 1 {\sigma^2}\, \partial_\gamma m_t(\theta)\, \big (X_t-m_t(\theta) \big ) \\
\frac 1 {2\, \sigma^4}\, \big ( \big (X_t-m_t(\theta) \big )^2-\sigma^2 \big ) 
\end{array}\right ),
\end{eqnarray}
and the same for $\partial_\theta \widehat q_t(\theta)$ by replacing $m_t(\theta)$ by $\widehat m_t(\theta)$. Once again for any $\theta \in \widetilde \Theta$, $(\partial_\theta q_t(\theta))_t$ is a stationary process, while $(\partial_\theta \widehat q_t(\theta))_t$ is not. Finally, for all $\theta\in \widetilde \Theta$, define
$$
\partial _\theta L_n(\theta)=\frac 1 n \, \sum_{t=1}^n \partial_\theta q_t(\theta) \quad\mbox{and} \quad \partial _\theta \widehat L_n(\theta)=\frac 1 n \, \sum_{t=1}^n \partial_\theta \widehat q_t(\theta).
$$
Following the same reasoning it can be shown that for any $t \in \Z$, $\theta \in \widetilde \Theta \mapsto q_t(\theta)$ and $\theta \in \widetilde \Theta \mapsto \widehat q_t(\theta)$ are a.s. ${\cal C}^2(\widetilde \Theta)$ functions and therefore the random matrices $ \partial^2 _{\theta^2} L_n(\theta)$ and $ \partial^2 _{\theta^2} \widehat L_n(\theta)$ a.s. exist.
~\\
The proof of Theorem \ref{theo2} will be decomposed in 3 parts:
\begin{enumerate}
\item First, as it was already established in Bardet and Wintenberger (2009), $(\partial_\theta  q_t(\theta^*))_t$ is a stationary ergodic martingale difference since with the $\sigma$-algebra ${\cal F}_t =\sigma\big \{ (X_{t-k})_{k\geq 1} \big \}$,
$$
\E \Big [ \partial_\theta  q_t(\theta^*)  ~\big | \, {\cal F}_t \Big ]=0,
$$
because $(X_t)$ is a causal process and $\varepsilon_t$ is independent of ${\cal F}_t$ and $\E \big [\varepsilon_0 ^2\big ]=1$.  \\
Now since  $\E \big [ \big \|\partial^{}_{\theta}  q_0(\theta^*) \big \|^2 \big ]<\infty$ from the same arguments as in the proof of the consistency of the estimator. 
Then the central limit for stationary ergodic martingale difference, Theorem 18.3 of Billingsley (1968)  can be applied and 
\begin{equation}\label{proof1}
\sqrt n \, \partial _\theta L_n(\theta^*) \limiteloin {\cal N} \big (0 \, , \, G^*\big ),
\end{equation}
since $\E \big [ \partial_\theta  q_0(\theta^*) \big ] =0$ and where $G^* :=\E \Big [ \partial_\theta q_0(\theta^*)  \times {}^t \big (  \partial_\theta q_0(\theta^*)\big )\Big ]$. 
\item We are going to prove that:
\begin{equation}\label{vardL}
n \, \E \Big [\sup_{\theta \in \widetilde \Theta}\big \| \partial_\theta \widehat{L}_n(\theta)-\partial_\theta {L}_n(\theta)\big \|^2 \Big ]\limiten 0.
\end{equation}
Using a line of reasoning already used in Beran and Sch\"utzner (2009, Lemma 1 and 2) and Bardet (2023, Lemma 5.1 3.), and derived from Parzen (1995, Theorem 3.B), there exists $C>0$ such that:
$$
\E \Big [\sup_{\theta \in \widetilde \Theta}\big \| \partial_\theta \widehat{L}_n(\theta)-\partial_\theta {L}_n(\theta)\big \|^2 \Big ] \leq C \, \sup_{\theta \in \widetilde \Theta} \E \Big [ \big \| \partial_\theta \widehat{L}_n(\theta)-\partial_\theta {L}_n(\theta)\big \|^2 \Big ],
$$
because we assumed that $\theta\to u_n(\theta)$ is a ${\cal C}^{p+1}(\widetilde \Theta)$ function and therefore $\partial_\theta \widehat{L}_n(\theta)-\partial_\theta {L}_n(\theta)$ is a ${\cal C}^{p}(\widetilde \Theta)$ function. \\
Then, for $\theta \in \widetilde \Theta$, 
$$
\partial_{\gamma}  q_t(\theta)- \partial_{\gamma} \widehat q_t(\theta)=\frac 1 {\sigma^2} \Big ( \partial_{\gamma} \widetilde m_t(\theta) \, \big (X_t -m_t(\theta)\big ) +\partial_{\gamma} \widehat m_t(\theta) \, \widetilde m_t(\theta) \Big ).
$$

As a consequence, for $\theta \in \widetilde \Theta$, 
\begin{multline}\label{decomp}
\!\! \!\! \! n \,\E \big [  \big \| \partial_\theta \widehat{L}_n(\theta)-\partial_\theta {L}_n(\theta)\big \|^2 \big ] \\
= \frac 1 {n \, \sigma^4}  \Big ( 2 \!\!\!\! \sum_{1 \leq s<t \leq n} \!\!\!\! \E \Big [  {}^t \Big ( \partial_{\gamma} \widetilde m_t(\theta) \, \big (X_t -m_t(\theta)\big ) +\partial_{\gamma} \widehat m_t(\theta) \, \widetilde m_t(\theta) \Big )\, \\
\times \Big ( \partial_{\gamma} \widetilde m_s(\theta) \, \big (X_s -m_s(\theta)\big ) +\partial_{\gamma} \widehat m_s(\theta) \, \widetilde m_s(\theta) \Big )\Big ] \\
~~~~ +  \sum_{t=1}^n \E \Big [{}^t \Big ( \partial_{\gamma} \widetilde m_t(\theta) \, \big (X_t -m_t(\theta)\big ) +\partial_{\gamma} \widehat m_t(\theta) \, \widetilde m_t(\theta) \Big )\, \\
\times \Big ( \partial_{\gamma} \widetilde m_t(\theta) \, \big (X_t -m_t(\theta)\big ) +\partial_{\gamma} \widehat m_t(\theta) \, \widetilde m_t(\theta) \Big ) \Big ] \Big )\\
=  \frac 1 {n \, \sigma^4}  \big (I_1+I_2 \big ).
\end{multline}
Concerning $I_1$, since $X_t=\sigma^*\, \varepsilon_t+m_t(\theta^*)$ and since $\varepsilon_t$ is independent to all the other terms because $s<t$, we deduce that  $ \big (X_t -m_t(\theta)\big )$ can be replaced by $n_t(\theta,\theta^*)= \big (m_t(\theta^*)-m_t(\theta) \big )$. As a consequence, after its expansion, $I_1$ can be written as a sum of $6$ expectations of products of $4$ linear combinations of $(\varepsilon_t)$.  Moreover, if for $j=1,\ldots,4$,  $Y^{(j)}_{t_j}=\sum_{k=0}^\infty \beta^{(j)}_k \, \xi_{t_j-k}$, where $t_1\leq t_2\leq t_3\leq t_4$, $(\beta^{(j)}_n)_{n\in \N}$ are $4$ real sequences and $(\xi_t)_{t\in \Z}$ is a white noise such as $\E[\xi_0^2]=1$ and $\E[\xi_0^4]=\mu_4< \infty$, then:
\begin{multline*}
\E \big [ \prod_{j=1}^4 Y^{(j)}_{t_j} \big ]=(\mu_4-3)\, \sum_{k=0}^\infty \beta^{(1)}_k \beta^{(2)}_{t_2-t_1+k}\beta^{(3)}_{t_3-t_1+k}\beta^{(4)}_{t_4-t_1+k} \\
+\E \big [Y^{(1)}_{t_1} Y^{(2)}_{t_2}\big ] \, \E \big [Y^{(3)}_{t_3} Y^{(4)}_{t_4}\big ] 
+ \E \big [Y^{(1)}_{t_1} Y^{(3)}_{t_3}\big ] \, \E \big [Y^{(2)}_{t_2} Y^{(4)}_{t_4}\big ] \\
+ \E \big [Y^{(1)}_{t_1} Y^{(4)}_{t_4}\big ] \, \E \big [  Y^{(2)}_{t_2} Y^{(3)}_{t_3} \big ].
\end{multline*}
Now, consider for example $Y^{(1)}_{t_1}=\partial_{\gamma} \widetilde m_s(\theta)$, $ Y^{(2)}_{t_2}= \big (X_s -m_s(\theta)\big )$, $Y^{(3)}_{t_3}=\partial_{\gamma} \widehat m_t(\theta) $ and $Y^{(4)}_{t_4}= \widetilde m_t(\theta)$.  From Lemma \ref{lem00} and for any used sequence $(\beta^{(j)}_k)_{k\in \N}$, there exists $C>0$ such as for any $k\in \N$:
\begin{multline*}
\big | \beta^{(1)}_k \big | \leq \frac C {s^d \, (k+1)^{1-d^*_+}}, ~\big | \beta^{(4)}_k \big | \leq \frac C {t^d \, (k+1)^{1-d^*_+}} \\
\mbox{and}~\max\big(\big | \beta^{(2)}_k\, , \, \big | \beta^{(3)}_k \big | \big )\leq  \frac C { (k+1)^{1-d^*_+}}.
\end{multline*}
As a consequence, with $s<t$,
\begin{multline}\label{mu4}
\Big | (\mu_4-3)\, \sum_{k=0}^\infty \beta^{(1)}_k \beta^{(2)}_{k}\beta^{(3)}_{t-s+k}\beta^{(4)}_{t-s+k}  \Big | \leq \frac C {s^d t^d} \, \sum_{k=1}^\infty \frac 1 {k^{2-2d^*_+}} \, \frac 1 {(k+t-s)^{2-2d^*_+}}\\ \leq  \frac C {s^d t^d(t-s)^{2-2d^*_+}}.
\end{multline}
And we obtain the same bound for any quadruple products appearing in $I_1$.  \\
~\\
Consider now the other terms of $I_1$. Using Lemmas \ref{lem2} and \ref{lem3}, we obtain for any $\theta\in \widetilde \Theta$ and $s<t$:
\begin{eqnarray*}
 \bullet \quad \Big |\E \big [Y^{(1)}_{t_1} Y^{(2)}_{t_2}\big ] \, \E \big [Y^{(3)}_{t_3} Y^{(4)}_{t_4}\big ]\Big |&=&\Big |\E \big [\partial_{\gamma} \widetilde m_s(\theta)\, \big (X_s -m_s(\theta)\big ) \big ] \, \E \big [ \partial_{\gamma} \widehat m_t(\theta) \, \widetilde m_t(\theta)\big ]\Big | \\
 & =&\Big | \E \big [\partial_{\gamma} \widetilde m_s(\theta)\, n_s(\theta,\theta^*)\big ) \big ]\Big | \, \Big |\E \big [\partial_{\gamma} \widehat m_t(\theta) \, \widetilde m_t(\theta)\big ]\Big | \\
& \leq & C\,  \frac 1 {s^{1+d-2d^*_+}}\, \frac 1 {t^{1+d-2d^*_+}}; \\
 \bullet \quad \Big |\E \big [Y^{(1)}_{t_1} Y^{(3)}_{t_3}\big ] \, \E \big [Y^{(2)}_{t_2} Y^{(4)}_{t_4}\big ]\Big |&=&\Big |\E \big [\partial_{\gamma} \widetilde m_s(\theta)\,  \partial_{\gamma} \widehat m_t(\theta)  \big ] \, \E \big [\big (X_s -m_s(\theta)\big ) \, \widetilde m_t(\theta)\big ]\Big | \\
 & =&\Big | \E \big [\partial_{\gamma} \widetilde m_s(\theta)\,\partial_{\gamma} \widehat m_t(\theta) \big ) \big ]\Big | \, \Big |\E \big [ n_s(\theta,\theta^*)\, \widetilde m_t(\theta)\big ]\Big | \\
& \leq & C\,  \Big ( \frac 1  {s^dt^{1-2d^*_+}}+ \frac {1} {s^{1+2d-2d^*_+}}\Big )\,\Big ( \frac 1 {t^{1+d} s^{-2d^*_+}} +\frac 1 {t^{1+2d-2d^*_+}}  \Big )\\
\bullet \quad \Big |\E \big [Y^{(1)}_{t_1} Y^{(4)}_{t_4}\big ] \, \E \big [Y^{(2)}_{t_2} Y^{(3)}_{t_3}\big ]\Big |&=&\Big |\E \big [\partial_{\gamma} \widetilde m_s(\theta)\, \widetilde m_t(\theta) \big ] \, \E \big [\big (X_s -m_s(\theta)\big ) \,  \partial_{\gamma} \widehat m_t(\theta)  \big ]\Big | \\
 & =&\Big | \E \big [\partial_{\gamma} \widetilde m_s(\theta)\, \widetilde m_t(\theta) \big ] \Big | \, \Big |\E \big [ n_s(\theta,\theta^*)\, \partial_{\gamma} \widehat m_t(\theta)\big ]\Big | \\
& \leq & C\,  \frac 1 {s^dt^{1-2d^*_++d}}\,\frac 1 {(t-s)^{1-2d^*_+} }
\end{eqnarray*}
Using these inequalities as well as \eqref{mu4}, we deduce from classical comparisons between sums and integrals:
\begin{eqnarray*}
&& \hspace{-1cm} \sum_{1 \leq s<t \leq n} \!\! \E \Big [  {}^t \Big ( \partial_{\gamma} \widetilde m_t(\theta) \, \big (X_t -m_t(\theta)\big ) \partial_{\gamma} \widehat m_s(\theta) \, \widetilde m_s(\theta) \Big )\Big ]  \\
 &\leq& C \!\! \sum_{1 \leq s<t \leq n} \!\!  \frac { \mu_4-3} {s^dt^d(t-s)^{2-2d^*_+}}+\Big ( \frac 1  {s^dt^{1-2d^*_+}}+ \frac {1} {s^{1+2d-2d^*_+}}\Big )\,\Big ( \frac 1 {t^{1+d} s^{-2d^*_+}} +\frac 1 {t^{1+2d-2d^*_+}}  \Big )
  \\
&&\hspace{3cm} +\frac 1 {s^{1+d-2d^*_+}}\, \frac 1 {t^{1+d-2d^*_+}} +\frac 1 {s^dt^{1-2d^*_++d}}\,\frac 1 {(t-s)^{1-2d^*_+} } \\
&\leq &C \Big ( \int_1^n x^{2d^*_+-1-2d} dx+\int_1^n \frac{dx}{x^{2+d-2d^*_+}} \int _1 ^x \frac{dy}{y^{d-2d^*_+}} \\
&&\hspace{1.5cm} +\int_1^n \frac{dx}{x^{2+2d-4d^*_+}} \int _1 ^x \frac{dy}{y^{d}}+\int_1^n \frac{dx}{x^{1+d}} \int _1 ^x \frac{dy}{y^{1+2d-4d^*_+}} \\
&&\hspace{1.5cm} +\int_1^n \frac{dx}{x^{1+2d-2d^*_+}} \int _1 ^x \frac{dy}{y^{1+2d-2d^*_+}} +\int_1^n \frac{dx}{x^{1+d-2d^*_+}} \int _1 ^x \frac{dy}{y^{1+d-2d^*_+}}  \\
&&\hspace{1.5cm}+ \int_1^n \frac{dx}{x^{1+d-2d^*_+}} \int _1 ^x \frac {dy} {y^d(x-y)^{1-2d^*_+} }  \Big ) \\
 &\leq &C \big (n^{2d^*_+-2d} + n^{4d^*_+-2d}+n^{4d^*_+-3d}+n^{4d^*_+-3d}+n^{4d^*_+-4d}+n^{4d^*_+-2d}+n^{4d^*_+-2d} \big ) \\
 &\leq&  C \,  n^{4d^*_+-2d}.
\end{eqnarray*}
We obtain exactly the same bounds if we consider the 3 others expectations, {\it i.e.}  $\E \Big [  {}^t \Big ( \partial_{\gamma} \widetilde m_t(\theta) \, \big (X_t -m_t(\theta)\big )\Big ) \partial_{\gamma} \widetilde m_s(\theta) \, \big (X_s -m_s(\theta)\big )\Big ]$,\\ $\E \Big [  {}^t \Big ( \partial_{\gamma} \widehat m_t(\theta) \, \widetilde m_t(\theta) \Big ) \partial_{\gamma} \widetilde m_s(\theta) \, \big (X_s -m_s(\theta)\big )\Big )\Big ]$ or $\E \Big [  {}^t \Big (\partial_{\gamma} \widehat m_t(\theta) \, \widetilde m_t(\theta) \Big )\partial_{\gamma} \widehat m_s(\theta) \, \widetilde m_s(\theta) \Big )\Big ]$. As a consequence, we finally obtain:
\begin{equation}\label{I1}
\frac 1 {\sigma^4 \, n} \, I_1 \leq C\, n^{4d^*_+-2d-1}\quad \mbox{for any $n\in \N^*$}. 
\end{equation}
Now consider the term $I_2$ in \eqref{decomp} and therefore the case $s=t$. For $Y^{(1)}_{t_1}=\partial_{\gamma} \widetilde m_t(\theta)$, $ Y^{(2)}_{t_2}= \big (X_t -m_t(\theta)\big )$, $Y^{(3)}_{t_3}=\partial_{\gamma} \widehat m_t(\theta) $ and $Y^{(4)}_{t_4}= \widetilde m_t(\theta)$, and the coefficient $(\beta^{(j)}_k)$ defined previously, we obtain:
\begin{multline}\label{I21}
\Big | (\mu_4-3)\, \sum_{k=0}^\infty \beta^{(1)}_k \beta^{(2)}_{k}\beta^{(3)}_{k}\beta^{(4)}_{k}  \Big | 
\leq C \, \sum_{k=1}^\infty \frac 1 {t^{2d}} \, \frac 1 {k^{4-4d^*_+}} \leq C \,\frac 1 {t^{2d}}.
\end{multline}
Moreover, using the same inequalities as in the case $s<t$, we obtain:
\begin{eqnarray*}
 \bullet \quad \Big |\E \big [Y^{(1)}_{t_1} Y^{(2)}_{t_2}\big ] \, \E \big [Y^{(3)}_{t_3} Y^{(4)}_{t_4}\big ]\Big |&\leq& C\, \frac 1 {t^{2+2d-4d^*_+}}; \\
 \bullet \quad \Big |\E \big [Y^{(1)}_{t_1} Y^{(3)}_{t_3}\big ] \, \E \big [Y^{(2)}_{t_2} Y^{(4)}_{t_4}\big ]\Big |
& \leq & C\,   \frac 1  {t^{2+2d-4d^*_+}}\\
\bullet \quad \Big |\E \big [Y^{(1)}_{t_1} Y^{(4)}_{t_4}\big ] \, \E \big [Y^{(2)}_{t_2} Y^{(3)}_{t_3}\big ]\Big |
& \leq & C\,  \frac 1 {t^{1-2d^*_++2d}}.
\end{eqnarray*}
Therefore, 
\begin{multline*}
 \sum_{t=1}^ n \E \Big [  {}^t \Big ( \partial_{\gamma} \widetilde m_t(\theta) \, \big (X_t -m_t(\theta)\big ) \partial_{\gamma} \widehat m_t(\theta) \, \widetilde m_t(\theta) \Big )\Big ] \\
  \leq C   \sum_{t=1}^ n \frac { \mu_4-3}{t^{2d}}+ \frac 1 {t^{1-2d^*_++2d}} \leq   C \, n^{1-2d}.
\end{multline*} 
As a consequence, we finally obtain that there exists $C>0$ such that:
\begin{equation}\label{I2}
\frac 1 {\sigma^4 \, n} \, I_2 \leq C\, n^{-2d}\quad \mbox{for any $n\in \N^*$.}
\end{equation}
Therefore, from \eqref{I1} and \eqref{I2}, we deduce that there exists $C>0$ such that for any $n\in \N^*$:
\begin{equation}\label{majo}
n \, \E \big [ \big \| \partial_\theta \widehat{L}_n(\theta)-\partial_\theta {L}_n(\theta)\big \|^2 \big ]\leq C\, \big (n^{-2d}+n^{4d^*_+-2d-1}\big ) \limiten 0,
\end{equation}
from \eqref{dd}. 
\item For $\theta\in \widetilde \Theta$ and $n \in \N^*$, since $ \partial^2 _{\theta^2} \widehat L_n(\theta)$ is a.s. a ${\cal C}^2(\widetilde \Theta)$ function, the Taylor-Lagrange expansion implies:
$$
\sqrt n\,\partial_\theta \widehat{L}_n(\theta^{*})=\sqrt n \, \partial_\theta \widehat{L}_n(\widetilde \theta_n)+\partial^{2}_{\theta^2} \widehat{L}_n(\bar{\theta}_{n}) \times \sqrt n \, (\theta^{*} - \widetilde{\theta}_{n}) 
$$
where $\bar{\theta}_{n}=c\,\widetilde{\theta}_{n}+(1-c)\,\theta^{*}$ and $0<c<1.$ But $\partial_\theta \widehat{L}_n(\widetilde \theta_n)=0$ because $\widetilde \theta_n$ is the unique local extremum of $\theta \to \widehat{L}_n(\theta)$. Therefore, 
\begin{equation}\label{proof2}
\sqrt n\,\partial_\theta \widehat{L}_n(\theta^{*})= \partial^{2}_{\theta^2} \widehat{L}_n(\bar{\theta}_{n}) \times \sqrt n \, (\theta^{*} - \widetilde{\theta}_{n}).
\end{equation}
Now,  $\E \big [ \big \|\partial^{2}_{\theta^2}  q_0(\theta) \big \| \big ]<\infty$ from the same arguments as in the proof of the consistency of the estimator, and  using Theorem 36.4 in Billingsley (1995),  $\big (\partial^{2}_{\theta^2}  q_t(\theta) \big )\big )_{t\in \Z}$ is an ergodic stationary sequence for any $\theta\in \widetilde \Theta$. Moreover $\bar \theta_n \limiteasn \theta^*$ since $\widetilde \theta_n \limiteasn \theta^*$. Hence:
$$
\partial^{2}_{\theta^2} {L}_n(\bar{\theta}_{n}) \limiteasn \E \big [ \partial^{2}_{\theta^2}  q_0(\theta) \big ]=F(\theta^*).
$$
Moreover, using the same arguments as in Lemma 4 of \cite{BW}, we have:
\begin{equation}\label{F2}
\sup_{\theta \in \widetilde \Theta} \Big \| \partial^{2}_{\theta^2} {L}_n(\theta)- \partial^{2}_{\theta^2} \widehat {L}_n(\theta)\Big \| \limiteproban 0 \quad \Longrightarrow \quad \partial^{2}_{\theta^2} \widehat  {L}_n(\bar{\theta}_{n}) \limiteproban F(\theta^*). 
\end{equation}
Usual calculations show that:
\begin{multline*}
F(\theta^*)=- \left ( \begin{array}{cc}M^* & 0 \\ 0 & \frac 1 {2 \, \sigma^{*4}} \end{array}\right )\quad \mbox{and}\quad G(\theta^{*})=\left ( \begin{array}{cc}M^* & 0 \\ 0 & \frac {\mu_4^*-1} {4 \, \sigma^{*4}} \end{array}\right ),\\
\mbox{with}\quad M^*=\frac 1 {\sigma^{*2}} \ \sum_{k=1}^\infty \sum_{\ell=1}^\infty \partial_\gamma u_k((\gamma^*,0))\, {}^t \big (\partial_\gamma u_\ell((\gamma^*,0)) \big )\, r_X(\ell-k)
\end{multline*}
where $G(\theta^{*})=\E \left[ \partial_\theta q_{0}(\theta^{*}) \, {}^t \partial_\theta q_{0}(\theta^{*}) \right]$ has already been defined in \eqref{proof1}. \\
Thanks to the formula for $M^*$, we can deduce that $F^*$ is invertible. Indeed, $M^*$ is invertible if and only if $\E \left[ \partial_\theta q_{0}(\theta^{*}) \, {}^t \partial_\theta q_{0}(\theta^{*}) \right]$ is invertible and therefore if and only if for all $v\in \R^{p-1}$, ${}^t v \, \E \left[ \partial_\gamma q_{0}(\theta^{*}) \, {}^t \partial_\theta q_{0}(\theta^{*}) \right]\, v=\E \left[ \big ({}^t v \,\partial_\gamma q_{0}(\theta^{*}) \big )^2 \right]=0$ or ${}^t v \,\partial_\gamma q_{0}(\theta^{*}) =0~a.s.$ implies $v=0$. Or, pour $v\in \R^{p-1}$,
\begin{eqnarray*}
 {}^t v \,\partial_\gamma q_{0}(\theta^{*})=0~~a.s. 
& \Longrightarrow & \frac 1 {\sigma^{*2}} \, \varepsilon_0 \, \sum_{k=1}^\infty {}^t v \,\partial _\gamma  u_k(\theta^*) \, X_{-k}=0~ ~a.s. \\
&& \hspace{-2.5cm} \Longrightarrow  \sum_{k=1}^\infty {}^t v \,\partial _\gamma  u_k(\theta^*) \, X_{-k}=0~ ~a.s.~~\mbox{($\varepsilon_0$ is independent to ${\cal F}_0$)} \\
&&\hspace{-2.5cm} \Longrightarrow {}^t v \,\partial _\gamma  u_k(\theta^*)=0\quad\mbox{for all $k\in \N^*$} \\
&& \hspace{-2.5cm} \Longrightarrow v=0\quad \mbox{from \eqref{Var}}.
\end{eqnarray*}
Now, from \eqref{proof1} and \eqref{majo}, we deduce that:
$$
\sqrt n\,\partial_\theta \widehat{L}_n(\theta^{*}) \limiteloin {\cal N} \big ( 0\, , \, G(\theta^{*}) \big),
$$
and since $F(\theta^*)$ is a definite negative matrix, from \eqref{proof2} we deduce that 
\begin{equation}\label{TLC}
\sqrt n \, \big (\widetilde \theta_n -\theta^*  \big )\limiteloin {\cal N} \big ( 0 \, , \,F(\theta^{*})^{-1}\,G(\theta^{*})\, F(\theta^{*})^{-1}  \big ).
\end{equation} 
Finally, from the previous computations of $G(\theta^{*})$ and $F(\theta^{*})$, we deduce \eqref{TLC2}.
\end{enumerate}
\end{proof}
\subsection{Proofs of additional lemmas}
\begin{lem}\label{lem00}
Under the assumptions of Theorem \ref{Theo1}, for any $\theta \in \Theta$ and $t\in \Z$ or $t\in \N^*$, with $m_t(\theta)$, $\widehat m_t(\theta)$ and $\widetilde m_t(\theta)$ respectively defined in \eqref{mt}, \eqref{mth} and \eqref{mtt}, we have:
$$
m_t(\theta)=\sum_{k=1}^\infty \alpha_k(\theta,\theta^*)\, \varepsilon_{t-k},~ \widehat m_t(\theta)=\sum_{k=1}^\infty \widehat \alpha_{k,t}(\theta,\theta^*)\, \varepsilon_{t-k}~\mbox{and}~ \widetilde m_t(\theta)=\sum_{k=0}^\infty \widetilde \alpha_{k,t}(\theta,\theta^*)\, \varepsilon_{-k},
$$
where there exists $C>0$ such as for any $k \geq 1$ and $t \in \N^*$,
$$
\max \big (\big |\alpha_k(\theta,\theta^*) \big |\, , \,\big |\widehat \alpha_{k,t}(\theta,\theta^*) \big | \big ) \leq  \frac C {k^{1-d^*_+}}\quad \mbox{and}\quad 
\big |\widetilde \alpha_{k,t}(\theta,\theta^*)\big |  \leq  \frac C {t^d \, k^{1-d^*_+}}.
$$
Moreover, under the assumptions of Theorem \ref{theo2}, the same properties also hold for $\partial_\theta m_t(\theta)$, $\partial_\theta \widehat m_t(\theta)$ and $\partial_\theta \widetilde m_t(\theta)$.
\end{lem}
\begin{proof}
We know that $X_t=\sum_{\ell=0}^\infty a_\ell(\theta^*)\, \varepsilon_{t-\ell}$ for any $t\in \Z$. Then,
\begin{eqnarray*}
m_t(\theta)&=&\sum_{k=1}^\infty\sum_{\ell=0}^\infty  u_k(\theta)a_\ell(\theta^*)\, \varepsilon_{t-k-\ell}=\sum_{j=1}^\infty \Big ( \sum_{k=1}^j u_k(\theta)a_{j-k}(\theta^*)\Big ) \,\varepsilon_{t-j}\\
&& \hspace{6cm} =\sum_{j=1}^\infty \alpha_j(\theta,\theta^*)\, \varepsilon_{t-j} \\
\widehat m_t(\theta)&=&\sum_{k=1}^{t-1} \sum_{\ell=0}^\infty  u_k(\theta)a_\ell(\theta^*)\, \varepsilon_{t-k-\ell}=\sum_{j=1}^\infty \Big (\!\!\!\! \sum_{k=1}^{\min(j\,,\, t-1)} \!\!\!\!\!\!\!\!u_k(\theta)a_{j-k}(\theta^*)\Big ) \,\varepsilon_{t-j}\\
&& \hspace{6cm} =\sum_{j=1}^\infty \widehat\alpha_{j,t}(\theta,\theta^*)\, \varepsilon_{t-j}  \\
\widetilde m_t(\theta)&=&\sum_{k=t}^{\infty} \sum_{\ell=0}^\infty  u_k(\theta)a_\ell(\theta^*)\, \varepsilon_{t-k-\ell} =\sum_{j=0}^\infty \Big ( \sum_{k=0}^j u_{t+k}(\theta)a_{j-k}(\theta^*)\Big ) \,\varepsilon_{t-j} \\
&& \hspace{6cm}=\sum_{j=0}^\infty \widetilde \alpha_{j,t}(\theta,\theta^*)\, \varepsilon_{t-j} 
\end{eqnarray*}
As a consequence, using $\big |a_\ell(\theta^*)\big |\leq C \, \ell ^{d^*_+-1}$  and $\big |u_\ell(\theta)\big |\leq C \, \ell ^{-d-1}$ for any $\ell\in \N^*$, we obtain:
\begin{eqnarray*}
\big |\alpha_j(\theta,\theta^*)\big | & \leq & C \, \sum_{k=1}^j \frac 1 {k^{1+d}} \, \frac 1 {(1+j-k)^{1-d^*_+}} \\
 & \leq & C \,\Big ( \frac 1 {(j/2)^{1-d^*_+}}\,  \sum_{k=1}^{j/2} \frac 1 {k^{1+d}}+ \frac 1 {(j/2)^{1+d}}\, \sum_{k=j/2} ^j \frac 1 {(1+j-k)^{1-d^*_+}} \Big ) \\
 & \leq & \frac C {j^{1-d^*_+}}.
\end{eqnarray*}
Using the same kind of decomposition, we obtain the other bounds.
\end{proof}

\begin{lem}\label{lem1}
For any $\alpha>1$, $\beta\in (0,1)$, there exists $C>0$ such as for any  $1\leq a$, 
\begin{eqnarray*}
&&I_\alpha(a)= \sum_{k=1}^\infty \frac 1 { k^\alpha \, (k+a)^\alpha} \leq \frac C {a^\alpha} \\ 
&&I_\alpha(a,b)= \sum_{k=1}^\infty \frac 1 {(k+a)^\alpha \, (k+b)^\alpha} \leq \frac C {a^{\alpha-1}\, b^\alpha} \quad\mbox{for any $b>a\geq 1$}\\
&& J_{\alpha,\beta}(0,a)=\sum_{k=1}^\infty \frac 1 {(k+a)^\alpha \, k^\beta } \leq \frac C {a^{\alpha+\beta-1}} \\
&& J_{\alpha,\beta}(a,0)=\sum_{k=1}^\infty \frac 1 {k^\alpha \, (k+a)^\beta} \leq \frac C {a^{\beta}} \\
&& J_{\alpha,\beta}(a,b)=\sum_{k=1}^\infty \frac 1 {(k+a)^\beta \, (k+b)^\alpha} \leq \frac C {a^{\beta}b^{\alpha-1}} \min \big (1\, , \,  \frac a b\big )  \quad\mbox{for any $b\geq 1$}
\end{eqnarray*}
\end{lem}
\begin{lem} \label{lem2}
Under the assumptions of Theorem \ref{Theo1}, there exists $C>0$ such as for any $\theta \in \Theta$ and $1 \leq s \leq t\leq n$,
\begin{equation}
\big | \E \big [ \widetilde m_s(\theta)\, \widetilde m_t(\theta)\big ] \big | \leq \frac C {s^dt^{1-2d^*_++d}}. 
\end{equation}
\end{lem}
\begin{proof}
Using the bounds of functions $I_{1+d}$ and $J_{1+d,1-2d}$ defined in Lemma \ref{lem1}, we obtain 
\begin{eqnarray}
\nonumber  \E \big [ \widetilde m_s(\theta)\, \widetilde m_t(\theta)\big ]   & =& \sum_{k=s}^\infty \sum_{\ell=t}^\infty u_k(\theta) \, u_\ell(\theta) \, r_X(t-s+k-\ell)  \\
\nonumber &  & \hspace{-3cm} \leq C \, \sum_{k=1}^\infty \sum_{\ell=1}^\infty \frac 1 {(s+k)^{1+d}} \, \frac 1 {(t+\ell)^{1+d}} \, \frac 1 {(1+|k-\ell|)^{1-2d^*_+}}\\
\nonumber &  & \hspace{-3cm} \leq C \, \Big ( \sum_{j=1}^\infty \frac 1 {(1+j)^{1-2d^*_+}} \, \sum_{\ell=1}^\infty \frac 1 {(\ell+s+j)^{1+d}(\ell+t)^{1+d}} + \sum_{k=1}^\infty \frac 1 {(k+s)^{1+d}(k+t)^{1+d}} \Big ) \\
\nonumber & & \hspace{1cm} + \sum_{j=1}^\infty \frac 1 {(1+j)^{1-2d^*_+}} \, \sum_{k=1}^\infty \frac 1 {(k+s)^{1+d}(k+t+j)^{1+d}}  \Big )\\
\nonumber & & \hspace{-3cm} \leq C \, \Big (I_{1+d}(s,t) +\sum_{j=1}^\infty \frac 1 {(1+j)^{1-2d^*_+}} \, \big ( I_{1+d}(s+j,t)+I_{1+d}(s,t+j) \big )\Big )\\
\nonumber & &\hspace{-3cm} \leq C \, \Big ( \frac 1 {s^d t^{1+d}}+\frac 1 {s^d} \, J_{1+d,1-2d^*_+}(0,t) + \frac 1 {t^{d+1}}\,   \sum_{j=1}^{t-s} \frac 1 {(1+j)^{1-2d^*_+}} \,  \frac 1 {(s+j)^{d}}  \\
\nonumber && \hspace{1cm} +\frac 1 {t^{d}}\,   \sum_{j=t-s}^{\infty} \frac 1 {(1+j)^{1-2d^*_+}} \,  \frac 1 {(s+j)^{1+d}}  \Big )\\ 
\nonumber & &\hspace{-3cm} \leq C \, \Big ( \frac 1 {s^d t^{1+d}}+\frac 1 {s^dt^{1-2d^*_++d}}  + \frac 1 {s^d t^{d+1}}\, (t-s+1)^{2d^*_+} +\frac 1 {t^{d}}\, J_{1+d,1-2d^*_+}(t-s,t) \Big )  \\ 
\nonumber & & \hspace{-3cm}  \leq C \, \Big (\frac 1 {s^d t^{1+d}}+\frac 1 {s^dt^{1-2d^*_++d}}  + \frac 1 {s^d t^{1-2d^*_++d}} +\frac 1 {t}\, J_{1+d,1-2d^*_+}(t-s,t) \Big ) \\
 \nonumber & & \hspace{-3cm}   \leq  \frac C {s^dt^{1-2d^*_++d}}.
\end{eqnarray}
\end{proof}
\begin{lem} \label{lem3}
Under the assumptions of Theorem \ref{Theo1}, there exists $C>0$ such as for any $\theta \in \Theta$ and any $1 \leq s$ and $1\leq t$,
\begin{equation}
\big | \E \big [ \widetilde m_s(\theta)\, m_t(\theta)\big ] \big | \leq \left \{ \begin{array}{ll} \displaystyle 
C \,\Big ( \frac 1  {s^dt^{1-2d^*_+}} +\frac {(1+t-s)^{2d^*_+}} {s^{1+2d}}\Big ) & \mbox{if $s\leq t$} \\ \displaystyle
C \, \Big ( \frac {t^{2d^*_+}} {s^{1+d} } +\frac {(1+s-t)^{2d^*_+} }{s^{1+d}t^d} \Big )  & \mbox{if $s\geq t$} 
\end{array} \right . .
\end{equation}
\end{lem}
\begin{proof}
\begin{eqnarray}
\nonumber  \E \big [ \widetilde m_s(\theta)\, m_t(\theta)\big ]   & =& \sum_{k=s}^\infty \sum_{\ell=1}^\infty u_k(\theta) \, u_\ell(\theta) \, r_X(t-s+k-\ell)  \\
\nonumber & \leq & C \, \sum_{k=1}^\infty \sum_{\ell=1}^\infty \frac 1 {(s+k)^{1+d}} \, \frac 1 {\ell^{1+d}} \, \frac 1 {(1+|t+k-\ell|)^{1-2d^*_+}}\\
\nonumber & \leq & C \, \Big ( \sum_{j=1}^\infty \frac 1 {(1+t+j)^{1-2d^*_+}} \, \sum_{\ell=1}^\infty \frac 1 {(\ell+s+j)^{1+d}\ell^{1+d}}  \\
\nonumber & & \hspace{0.3cm} + \sum_{j=1}^t \frac 1 {(1+t-j)^{1-2d^*_+}} \, \sum_{k=1}^\infty \frac 1 {(k+s)^{1+d}(k+j)^{1+d}}  \\
\nonumber & & \hspace{1 cm} + \sum_{j=t}^\infty \frac 1 {(1+j-t)^{1-2d^*_+}} \, \sum_{k=1}^\infty \frac 1 {(k+s)^{1+d}(k+j)^{1+d}}  \Big )\\
\nonumber & \leq & C \, \Big ( \sum_{j=1}^\infty \frac 1 {(t+j)^{1-2d^*_+}} \,I_{1+d}(s+j,0) + \sum_{j=1}^t \frac 1 {j^{1-2d^*_+}} \,I_{1+d}(s,t-j) \\
&& \nonumber \hspace{6cm}   +
\sum_{j=1}^\infty \frac 1 {j^{1-2d^*_+}} \,I_{1+d}(s,j+t)\Big ).
\end{eqnarray}
Then, if $s \leq t$, 
\begin{eqnarray}
\nonumber  \big |\E \big [ \widetilde m_s(\theta)\, m_t(\theta)\big ]\big |  & \leq & C \, \Big ( J_{1+d,1-2d^*_+}(t,s) + \frac 1 {s^d} \, \sum_{j=1}^{t-s} \frac 1 {j^{1-2d^*_+}} \,\frac 1 {(t-j)^{1+d}} \\
\nonumber  && \hspace{0.3cm}  + \frac 1 {s^{1+d}} \, \sum_{j=1}^{s} \frac 1 {(t-j)^{1-2d^*_+}} \,\frac 1 {(s-j)^{1+d}}+\frac 1 {s^d} \, J_{1+d,1-2d^*_+}(0,t) \Big ) \\
\nonumber & \leq & C \, \Big ( \frac 1 {s^dt^{1-2d^*_+}} + \frac {1} {s^{2d}t^{1-2d^*_+}}+ \frac {1} {s^{1+2d-2d^*_+}}+\frac 1 {s^d t^{1+d-2d^*_+}}   \Big ) \\
\nonumber &\leq & C \Big ( \frac 1  {s^dt^{1-2d^*_+}}+ \frac {1} {s^{1+2d-2d^*_+}}\Big )  .
\end{eqnarray}
And if $s > t$, 
\begin{eqnarray}
\nonumber \big |\E \big [ \widetilde m_s(\theta)\, m_t(\theta)\big ]\big |  & \leq & C \, \Big ( J_{1+d,1-2d^*_+}(t,s) + \frac 1 {s^{1+d}}\, \sum_{j=1}^{t} \frac 1 {j^{1-2d^*_+}} \,\frac 1 {(t-j)^{1+d}} \\
\nonumber  && \hspace{0.3cm}  + \frac 1 {s^{1+d}} \, \sum_{j=1}^{s-t} \frac 1 {j^{1-2d^*_+}} \,\frac 1 {(t+j)^{d}}+\frac 1 {s^{d}} \, \sum_{j=s-t} ^\infty \frac 1 {j^{1-2d^*_+}} \,\frac 1 {(t+j)^{1+d}} \Big ) \\
\nonumber & \leq & C \, \Big ( \frac 1 {s^{1+d} t^{-2d^*_+}} + \frac 1 {s^{1+d}t^{1-2d^*_+}}+ \frac {1}{s^{1+2d-2d^*_+}}+\frac 1 {s^{1+2d-2d^*_+}}  \Big ) \\
\nonumber & \leq & C \, \Big ( \frac 1 {s^{1+d} t^{-2d^*_+}} +\frac 1 {s^{1+2d-2d^*_+}}  \Big ) .
\end{eqnarray}
\end{proof}

\paragraph*{Aknowledgments} The authors are grateful to the referees  for many relevant suggestions and comments that helped to notably improve and enrich the contents of the paper.

\end{document}